\documentclass{amsart}

\usepackage{amsmath}
\usepackage{amssymb}
\usepackage{geometry}
\usepackage{color}
\usepackage{graphicx}
\usepackage{verbatim}

\numberwithin{equation}{section} 

\newtheorem{theorem}{Theorem}[section]
\newtheorem{corollary}[theorem]{Corollary}
\newtheorem{lemma}[theorem]{Lemma}
\newtheorem{proposition}[theorem]{Proposition}
\theoremstyle{definition} 
\newtheorem{definition}[theorem]{Definition}
\newtheorem{example}{Example}[section]

\newtheorem{remark}[theorem]{Remark}
\newtheorem{assumptions}{Assumptions}
\newtheorem*{ack}{Acknowledgments}

\usepackage[colorlinks=true,urlcolor=blue, citecolor=red,linkcolor=blue,
linktocpage,pdfpagelabels, bookmarksnumbered,bookmarksopen]{hyperref}
\usepackage[hyperpageref]{backref}

\title{On the spectrum of sets made of cores and tubes}

\author[Bianchi]{Francesca Bianchi}
\address[F.\ Bianchi]{Dipartimento di Scienze Matematiche, Fisiche e Informatiche
	\newline\indent
	Universit\`a di Parma
	\newline\indent
	Parco Area delle Scienze 53/a, Campus, 43124 Parma, Italy}
\email{francesca.bianchi@unipr.it}

\author[Brasco]{Lorenzo Brasco}
\address[L.\ Brasco]{Dipartimento di Matematica e Informatica
	\newline\indent
	Universit\`a degli Studi di Ferrara
	\newline\indent
	Via Machiavelli 35, 44121 Ferrara, Italy}
\email{lorenzo.brasco@unife.it}

\author[Ognibene]{Roberto Ognibene}
\address[R.\ Ognibene]{Dipartimento di Matematica
	\newline\indent
	Universit\`a di Pisa
	\newline\indent
	Largo Bruno Pontecorvo 5, 56127 Pisa, Italy}
\email{roberto.ognibene@dm.unipi.it}

\date{\today}
\dedicatory{Dedicated to Giuseppe Buttazzo, a master of Calculus of Variations}

\subjclass[2010]{35P15, 35B38, 58E05}
\keywords{Eigenvalue estimates, Poincar\'e inequality, inradius, Makai-Hayman inequality, Palais-Smale sequence, curved waveguide.}

\begin{document}

\begin{abstract}
We analyze the spectral properties of a particular class of unbounded open sets. These are made of a central bounded ``core'', with finitely many unbounded tubes attached to it. We adopt an elementary and purely variational point of view, studying the compactness (or the defect of compactness) of level sets of the relevant constrained Dirichlet integral.
As a byproduct of our argument, we also get exponential decay at infinity of variational eigenfunctions. Our analysis includes as a particular case a planar set (sometimes called ``bookcover''), already encountered in the literature on curved quantum waveguides. J. Hersch suggested that this set could provide the sharp constant in the {\it Makai-Hayman inequality} for the bottom of the spectrum of the Dirichlet-Laplacian of planar simply connected sets. We disprove this fact, by means of a singular perturbation technique. 
\end{abstract}
	
\maketitle

\begin{center}
\begin{minipage}{10cm}
\small
\tableofcontents
\end{minipage}
\end{center}

\section{Introduction}
\subsection{Constrained critical points of Dirichlet integrals}
For an open set $\Omega\subseteq\mathbb{R}^N$, we indicate by $\mathcal{S}_2(\Omega)$ the sphere
\[
\mathcal{S}_2(\Omega)=\Big\{\varphi\in W^{1,2}_0(\Omega)\, :\, \|\varphi\|_{L^2(\Omega)}=1\Big\}.
\]
Here, by $W^{1,2}_0(\Omega)$ we mean the closure of $C^\infty_0(\Omega)$ in the standard Sobolev space $W^{1,2}(\Omega)$.
Then, we define
\begin{equation*}
\lambda_1(\Omega):=\inf_{\varphi\in \mathcal{S}_2(\Omega)} \int_\Omega |\nabla \varphi|^2\,dx.
\end{equation*}
This is nothing but the sharp constant in the Poincar\'e inequality for functions in $W^{1,2}_0(\Omega)$. Of course, we have $\lambda_1(\Omega)=0$ each time that $\Omega$ does not support such an inequality. 
Necessary and sufficient conditions on $\Omega$ assuring that $\lambda_1(\Omega)>0$ can be found for example in \cite[Chapter 15, Section 4]{Maz}.
\par
Even when $\Omega$ supports the Poincar\'e inequality, the value $\lambda_1(\Omega)$ is not necessarily attained, i.e. it may be just an infimum. A sufficient condition for $\lambda_1(\Omega)$ to be a minimum is the compactness of the embedding 
\begin{equation}
\label{embedding}
	W^{1,2}_0(\Omega)\hookrightarrow L^2(\Omega).
\end{equation}
In this case, there exists a minimizer $u_1\in\mathcal{S}_2(\Omega)$ and by optimality it solves in weak sense
\begin{equation}
\label{eigenequation}
-\Delta u_1=\lambda\,u_1,\qquad \mbox{ in }\Omega,
\end{equation}
with $\lambda=\lambda_1(\Omega)$, i.e. $\lambda_1(\Omega)$ is an eigenvalue of the Dirichlet-Laplacian on $\Omega$ and $u_1$ is an associated eigenfunction. Observe that we can always suppose that $u_1$ is non-negative, since both the constraint and the functional are invariant by the change $\varphi\mapsto |\varphi|$.
\par
However, in this situation, much more is true: the compactness of the embedding \eqref{embedding} entails that we can construct recursively 
a diverging sequence $\{\lambda_j(\Omega)\}_{j\in\mathbb{N}\setminus\{0\}}$ of positive numbers, each one being an eigenvalue of the Dirichlet-Laplacian on $\Omega$. Moreover, 
if we define\footnote{The constraint $u\in\mathcal{S}_2(\Omega)$ has no real bearing, since the equation \eqref{eigenequation} is linear. This is just a cheap way of saying that \eqref{eigenequation} admits a non-trivial solution.}
\[
\mathrm{Spec}(\Omega)=\Big\{\lambda\, :\, \eqref{eigenequation} \mbox{ admits a weak solution } u\in\mathcal{S}_2(\Omega)\Big\},
\]
we have that 
\[
\mathrm{Spec}(\Omega)=\Big\{\lambda_j(\Omega)\Big\}_{j\in\mathbb{N}\setminus\{0\}}.
\]
Finally, each $\lambda_j$ has the following variational characterization
\begin{equation}
\label{lambdak}
	\lambda_j(\Omega)=\inf\left\{ \max_{u\in F\cap \mathcal{S}_2(\Omega)} \int_\Omega|\nabla u|^2\,dx\,:\, F \subseteq W^{1,2}_0(\Omega) \mbox{ subspace with } \dim F=j \right\}.
\end{equation}
We refer for example to \cite[Chapter VI]{CH} for these facts.
\par
It is noteworthy to notice that, according to the Lagrange's multipliers rule, each element of $\mathrm{Spec}(\Omega)$ can be understood as a critical point of the Dirichlet integral
\[
\varphi\mapsto\int_\Omega |\nabla \varphi|^2\,dx,
\]
constrained to the ``manifold'' $\mathcal{S}_2(\Omega)$. In this interpretation, the associated eigenfunctions are the relevant critical points. The first eigenvalue $\lambda_1(\Omega)$ and an associated eigenfunction $u_1$ correspond to the global constrained minimum and a global minimizer, respectively. 
\par
Thus, the discussion above entails that when the embedding \eqref{embedding} is compact, the Critical Point Theory of this constrained functional is completely clear: there is only a discrete sequence of critical values, which coincides with $\{\lambda_j(\Omega)\}_{j\in\mathbb{N}\setminus\{0\}}$. Actually, a much more sophisticated conclusion can be drawn in this ``compact'' situation.
This is better elucidated by recalling the following definition, which is well-known in Critical Point Theory (see for example \cite[Chapter II, Section 2]{St}).
\begin{definition}
\label{def:PS}
Let $\lambda\ge 0$, we say that $\{\varphi_n\}_{n\in\mathbb{N}}\subseteq W^{1,2}_0(\Omega)$ is a {\it constrained Palais-Smale sequence at the level $\lambda$} if the following three properties hold:
\begin{enumerate}
\item $\varphi_n\in\mathcal{S}_2(\Omega)$, for every $n\in\mathbb{N}$;
\vskip.2cm
\item $\lim\limits_{n\to\infty}\|\nabla \varphi_n\|_{L^2(\Omega)}^2=\lambda$;
\vskip.2cm
\item we have\footnote{For every $u\in W^{1,2}_0(\Omega)$, we set
\[
\|-\Delta u-\lambda\,u\|_{W^{-1,2}(\Omega)}=\sup_{\varphi\in C^\infty_0(\Omega)} \left\{\left|\int_\Omega \langle \nabla u,\nabla \varphi\rangle\,dx-\lambda\,\int_\Omega u\,\varphi\,dx\right|\,:\,\|\varphi\|_{W^{1,2}(\Omega)}=1\right\}.
\]
}
\[
\lim_{n\to\infty} \|-\Delta \varphi_n-\lambda\,\varphi_n\|_{W^{-1,2}(\Omega)}=0.
\]
\end{enumerate}
\end{definition}
Then, as a consequence of the previous discussion, the only values $\lambda$ admitting a constrained Palais-Smale sequence are given by $\{\lambda_j(\Omega)\}_{j\in\mathbb{N}\setminus\{0\}}$. This is due to fact that if \eqref{embedding} is compact, then automatically a constrained Palais-Smale sequence would converge in $L^2(\Omega)$ (up to a subsequence). By condition (3), the limit function $u\in \mathcal{S}_2(\Omega)$ would be an eigenfunction and $\lambda$ an eigenvalue.

\subsection{Loss of compactness} What happens when the embedding \eqref{embedding} is no more compact? What can be said about the structure of critical values for the constrained Dirichlet integral, in this case? Is it possible to relate the behaviour of critical values to the ``defect of compactness'' of the embedding \eqref{embedding}?
\par
In this paper, we will try to (partially) answer these questions, in a particular class of open sets where the loss of compactness is controlled, in a suitable sense. Before presenting the class of open sets we want to deal with, we wish to make a couple of further observations.
\par
At first, we go back to the global infimum $\lambda_1(\Omega)$. Then we observe that the compactness of the embedding is certainly a too strong condition for this to be a minimum.
Indeed, it would be sufficient to know that a certain constrained sublevel set
\[
E_\Omega(t):=\left\{\varphi\in \mathcal{S}_2(\Omega)\, :\, \int_\Omega |\nabla \varphi|^2\,dx\le t\right\},\qquad \mbox{ with } t>\lambda_1(\Omega),
\]
is relatively compact in $L^2(\Omega)$, to infer the existence of a global minimizer $u_1$ (and thus of a first eigenfunction). In other words, if compactness of the embedding \eqref{embedding} is lost only for ``large energy levels'', we can still infer compactness of minimizing sequences.
\par
More generally, one can prove that if we define $\lambda_j(\Omega)$ by \eqref{lambdak} and compactness still holds for some $E_\Omega(t)$ with $t>\lambda_j(\Omega)$, then we still have room to prove that $\lambda_j(\Omega)$ actually defines a critical value (and thus an eigenvalue of the Dirichlet-Laplacian). 
\par
However, even these weaker conditions are in a certain sense too strong. By appealing to Definition \ref{def:PS}, it would be sufficient to know that $\lambda_j(\Omega)$ defined by \eqref{lambdak} admits a constrained Palais-Smale sequence (see above), enjoying a suitable form of compactness. More precisely, it would suffice that such a sequence admits a {\it weakly convergent subsequence} in $L^2(\Omega)$ to a limit function $u\not =0$. Indeed, even if this is not enough to assures that $u$ belongs to the constraint $\mathcal{S}_2(\Omega)$, nevertheless by linearity of $\varphi\mapsto -\Delta \varphi-\lambda\,\varphi$ we would get that $u$ is a non-trivial weak solution of \eqref{eigenequation}. Accordingly, we could conclude again that $\lambda_j(\Omega)$ is an eigenvalue of the Dirichlet-Laplacian on $\Omega$.
\par 
{\it \c{C}a va sans dire}, this compactness requirement on constrained Palais-Smale sequences is much weaker than the compactness of the sublevel sets $E_\Omega(t)$.

\subsection{A class of unbounded sets}
We are interested in open sets made of a ``central core'' $\Omega_0$, to which we attach a finite number of infinite cylindrical sets $\mathcal{C}_1,\dots,\mathcal{C}_k$. We admit that these cylindrical sets may have a non-empty pairwise intersection, provided the latter is bounded,  while we allow their boundaries to arbitrarily intersect. This class of sets is a generalization of those considered in \cite{BGRS}, among others.
\par
More precisely, throughout the whole paper we will make the following   

\begin{assumptions}\label{ass:Omega}
	For $N\ge 2$, let $\Omega\subseteq\mathbb{R}^N$ be an open set such that
	\begin{equation*}
		\Omega=\Omega_0\cup \bigcup_{i=1}^k \mathcal{C}_i,
	\end{equation*}
	where $k\in\mathbb{N}\setminus\{0\}$ and
	\begin{enumerate}
		\item[(A1)] $\Omega_0\subseteq\mathbb{R}^N$ is an open bounded set (it could be empty, see Example \ref{exa:cross});
		\vskip.2cm
		\item[(A2)] for $i=1,\dots,k$, the open set $\mathcal{C}_i$ is a cylindrical set given by
		\[		
		\mathcal{C}_{i}:=\textbf{R}_i(E_i\times(0,+\infty))+\mathbf{t}_i,
		\]
		where:
		\begin{itemize}
			\item	 $E_i\subseteq \mathbb{R}^{N-1}$ is an open bounded connected set;
			\vskip.2cm
			\item $\textbf{R}_i$ is a linear isometry represented by an $N\times N$ orthogonal matrix; 
			\vskip.2cm
			\item $\textbf{t}_i\in\mathbb{R}^N$ is a given point;
		\end{itemize}
		\vskip.2cm
		\item[(A3)] there exists $\varrho_0>0$ such that 
		\[
		\Big(\mathcal{C}_{i}\cap \mathcal{C}_j\Big)\setminus B_{\varrho_0}(0)=\emptyset,\qquad \mbox{ for every } i\not=j.
		\]
	\end{enumerate}
	\begin{figure}
		\includegraphics[scale=.25]{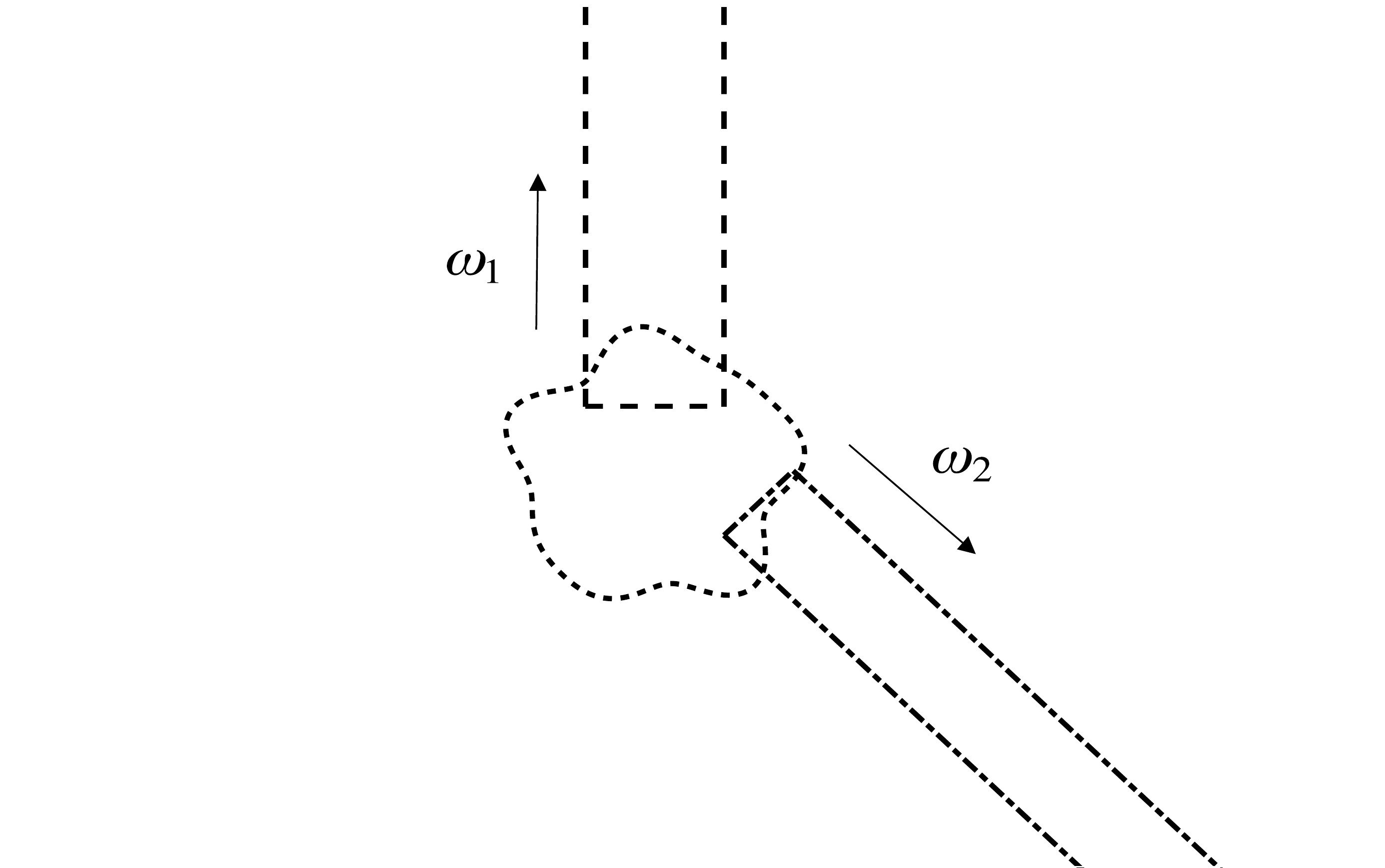}
		\caption{An example of set verifying Assumptions \ref{ass:Omega}, in dimension $N=2$ and with $k=2$ disjoint tubes $\mathcal{C}_1$ and $\mathcal{C}_2$.}	
		\label{fig:insieme2d}
	\end{figure}
	If we set $\mathbf{e}_N=(0,\dots,0,1)$, we also denote 
	\[
	\omega_i:=\textbf{R}_i(\mathbf{e}_N),\qquad \mbox{ for } i=1,\dots,k,
	\]
	which coincides with the direction of the axis of the cylindrical set $\mathcal{C}_i$ (see Figure \ref{fig:insieme2d}).
\end{assumptions}
For these sets, we will prove (see Theorem \ref{teo:first_eigen}) that if 
\[
\lambda_1(\Omega)<\mathcal{E}(\Omega):=\min\Big\{\lambda_1(E_1),\dots,\lambda_1(E_2)\Big\},
\]
then there exists a minimizing sequence
\[
\{u_n\}_{n\in\mathbb{N}}\subseteq\mathcal{S}_2(\Omega),\qquad \lim_{n\to\infty} \int_\Omega |\nabla u_n|^2=\lambda_1(\Omega),
\] 
which is compact in $L^2(\Omega)$. This assures that $\lambda_1(\Omega)$ is attained and there exists a first eigenfunction.
\par
In order to prove this, we will proceed as follows: we ``truncate'' $\Omega$ along the tubes, so to build a sequence of bounded sets $\{\Omega_n\}_{n\in\mathbb{N}}$ exhausting $\Omega$. Then, as a minimizing sequence we will take the sequence made of the first eigenfunctions of these sets, let us call it $\{u_1^n\}_{n\in\mathbb{N}}$. To get compactness of this sequence, we will rely on the equation satisfied by $u_1^n$. Indeed, this 
permits to get a Caccioppoli--type inequality ``localized at infinity'' along the tubes. We then join this estimate with Poincar\'e's inequality along the tubes $\mathcal{C}_i$, for which the relevant constant is given by\footnote{We recall that $\lambda_1$ is invariant by orthogonal transformations and translations, thus 
 we have 
\[
\lambda_1(\mathcal{C}_i)=\lambda_1(E_1\times (0,+\infty))=\lambda_1(E_1).
\]} $\lambda_1(E_i)$. By taking into account that $\lambda_1(\Omega)<\mathcal{E}(\Omega)$, we will get that $\{u_1^n\}_{n\in\mathbb{N}}$ must have a uniformly small $L^2$ norm at infinity. 
This finally gives the desired compactness, by means of the classical Riesz-Fr\'echet-Kolmogorov theorem. 
\par
In a nutshell: since the global infimum $\lambda_1(\Omega)$ is strictly less than the ``energy'' of all the tubes, minimizing sequences does not want to ``occupy too much'' a tube. Indeed, this would raise too much the value of the Dirichlet integral.
Rather, they try to concentrate as much as possible towards the ``core'' $\Omega_0$. This gives a gain of compactness.
\par
We point out that, as a direct byproduct of this proof, we get that the $L^2$ norm of a first eigenfunction for $\Omega$ must decay exponentially to $0$, at infinity. In turn, by means of a standard $L^\infty-L^2$ estimate ``localized at infinity'' (obtained for example by Moser's iteration), we can upgrade this exponential decay to a pointwise one (see Theorem \ref{teo:eigenproperties}).
\par
Actually, there is nothing specific to $\lambda_1$ in the previous argument.
More generally, if for every $j\in\mathbb{N}\setminus\{0\}$ we define $\lambda_j(\Omega)$ by \eqref{lambdak},
we can prove that whenever we have
\[
\lambda_j(\Omega)<\mathcal{E}(\Omega),
\]
then $\lambda_j(\Omega)$ is an eigenvalue (see Theorem \ref{teo:higher_eigen}). The argument is essentially the same as above: we rely again on the equations for the first $j-$eigenfunctions of the truncated sets $\Omega_n$.
\vskip.2cm\noindent
From this argument, we see that $\mathcal{E}(\Omega)$ plays an important role: it can be understood as the energy threshold under which some compactness survives and above which compactness {\it may completely fail}. 
We will make this precise and complement our analysis by showing in Proposition \ref{prop:PS} that {\it for every} 
\[
\lambda\ge \mathcal{E}(\Omega),
\]
there exists a sequence of ``almost critical points'' for the Dirichlet integral constrained to the sphere $\mathcal{S}_2(\Omega)$, for which compactness is completely lost, i.e. it weakly converges in $L^2(\Omega)$ to $0$. Thus, it is impossible to retrieve a critical point associated to $\lambda$ from such a sequence.  
More precisely, by using the definitions recalled above, we will show for every such $\lambda$ there exists a constrained Palais-Smale sequence weakly converging to $0$. We will construct this sequence ``by hands'', exploiting the geometry of our sets.
\par
We can build a bridge between Spectral Theory and Critical Point Theory: indeed, these particular sequences are nothing but
a variational reformulation of {\it singular Weyl sequences}, appearing in Spectral Theory, see for example \cite[Chapter 9, Section 2]{BS} and \cite[Chapter 6, Section 4]{Te}. These sequences are important since they permit to characterize the {\it essential spectrum} of a self-adjoint operator, see for example \cite[Theorem 9.2.2]{BS}. In this way, we retrieve a variational characterization of the essential spectrum of the Dirichlet-Laplacian on our sets.

\subsection{A detour: the Makai-Hayman inequality}
In order to neatly justify our original interest towards this class of sets, it is necessary to take a small {\it detour}.
Thus, let us now briefly turn our attention to an apparently unrelated problem: the so-called {\it Makai-Hayman inequality} in Spectral Geometry.
At this aim, we need to recall the definition of {\it inradius} $r_\Omega$ of an open set $\Omega\subseteq\mathbb{R}^N$. This is the quantity given by
\begin{equation*}
	r_\Omega:=\sup\Big\{ r>0\colon \text{there exists }x\in\Omega~\text{such that }B_r(x)\subseteq\Omega \Big\}.
\end{equation*}
A remarkable result due to Makai (see \cite{Mak}) asserts that there exists a universal constant $C>0$ such that
\begin{equation}
\label{MakHay}
\lambda_1(\Omega)\ge \frac{C}{r_\Omega^2},
\end{equation}
for every $\Omega\subseteq\mathbb{R}^2$ open {\it simply connected set} with finite inradius. The same result was then rediscovered independently by Hayman in \cite{Ha}, by means of a different proof. For this reason, we will call \eqref{MakHay} the Makai-Hayman inequality. 
\par
For completeness, we mention that this result can be extended to more general planar open sets, having nontrivial topology. These generalizations are contained in a series of papers appeared shortly after Hayman's one. We cite for example Croke \cite{Cr}, Graversen and Rao \cite{GR}, Osserman \cite{Os} and Taylor \cite{Ta}. More recently, some of these results have been generalized to the $p-$Laplacian by Poliquin (see \cite{Po}) and to the fractional Laplacian by the first two authors (see \cite{BiBr2, BiBr1}) .
\vskip.2cm\noindent 
The validity of \eqref{MakHay} immediately leads to a natural question: what is the sharp constant in such an inequality? 
In other words, if for every $\Omega\subseteq\mathbb{R}^2$ open simply connected set with finite inradius, we define the scale invariant quantity
\begin{equation}
\label{rhoberto}
\rho(\Omega):=r_\Omega^2\,\lambda_1(\Omega),
\end{equation}
we want to determine the value of the following spectral optimization problem
\[
C_{MH}=\inf\Big\{\rho(\Omega)\,:\, \Omega\subseteq \mathbb{R}^2 \mbox{ open simply connected with }r_\Omega<+\infty\Big\}.
\]
In his paper \cite{Mak}, Makai showed that 
\[
\frac{1}{4}\le C_{MH}<\frac{\pi^2}{4}.
\]
It is noteworthy to notice that
\[
\frac{\pi^2}{4}=\inf\Big\{\rho(\Omega)\,:\, \Omega\subseteq \mathbb{R}^2 \mbox{ open convex set with }r_\Omega<+\infty\Big\},
\]
i.e. this is the sharp constant in \eqref{MakHay}, in the restricted class of convex sets. Such a value is attained for example by infinite strips.
We recall that the determination of the sharp constant for convex sets is a celebrated result due to Hersch, see \cite[Th\'eor\`eme 8.1]{He}. 
\par
In particular, we see that relaxing ``convexity'' to ``simple connectedness'' certainly lower the value of the optimal constant.
However, the exact determination of $C_{MH}$ is still an open problem: at present, the best known result is that
\[
0.6197 < C_{MH}<2.095. 
\] 
The lower bound is due to Ba\~nuelos and Carroll (see \cite[Corollary 6.1]{BC}), while the upper bound has been proved by Brown (see \cite[Section 5]{Br}), by slightly improving a method used in \cite{BC}. After these results, essentially no progress has been made on the problem. We refer the reader to \cite{Bu} and \cite{Hen} for a comprehensive overwiev of results in spectral optimization.

\subsection{An example by Hersch} 
We can now explain our interest towards sets having the structure encoded by Assumptions \ref{ass:Omega}.
Indeed, in his review of Makai's paper (see \cite{HeRev}), Hersch suggested that an optimal set providing the sharp value $C_{MH}$ could be the following 
\begin{equation}
\label{pipe}
	\mathbf{H}:=\Big\{(x_1,x_2)\in\mathbb{R}^2\,:\, x_1^2+x_2^2<1,\, x< 0\Big\}\cup\Big\{(x_1,x_2)\in\mathbb{R}^2\,:\, 0<|x_2|<1,\, x_1\geq 0 \Big\},
\end{equation}
see Figure \ref{fig:hersch}. We will call it {\it Hersch's pipe}. This is a slit disk, to which two infinite tubes of constant width are attached.
\begin{figure}
\includegraphics[scale=.2]{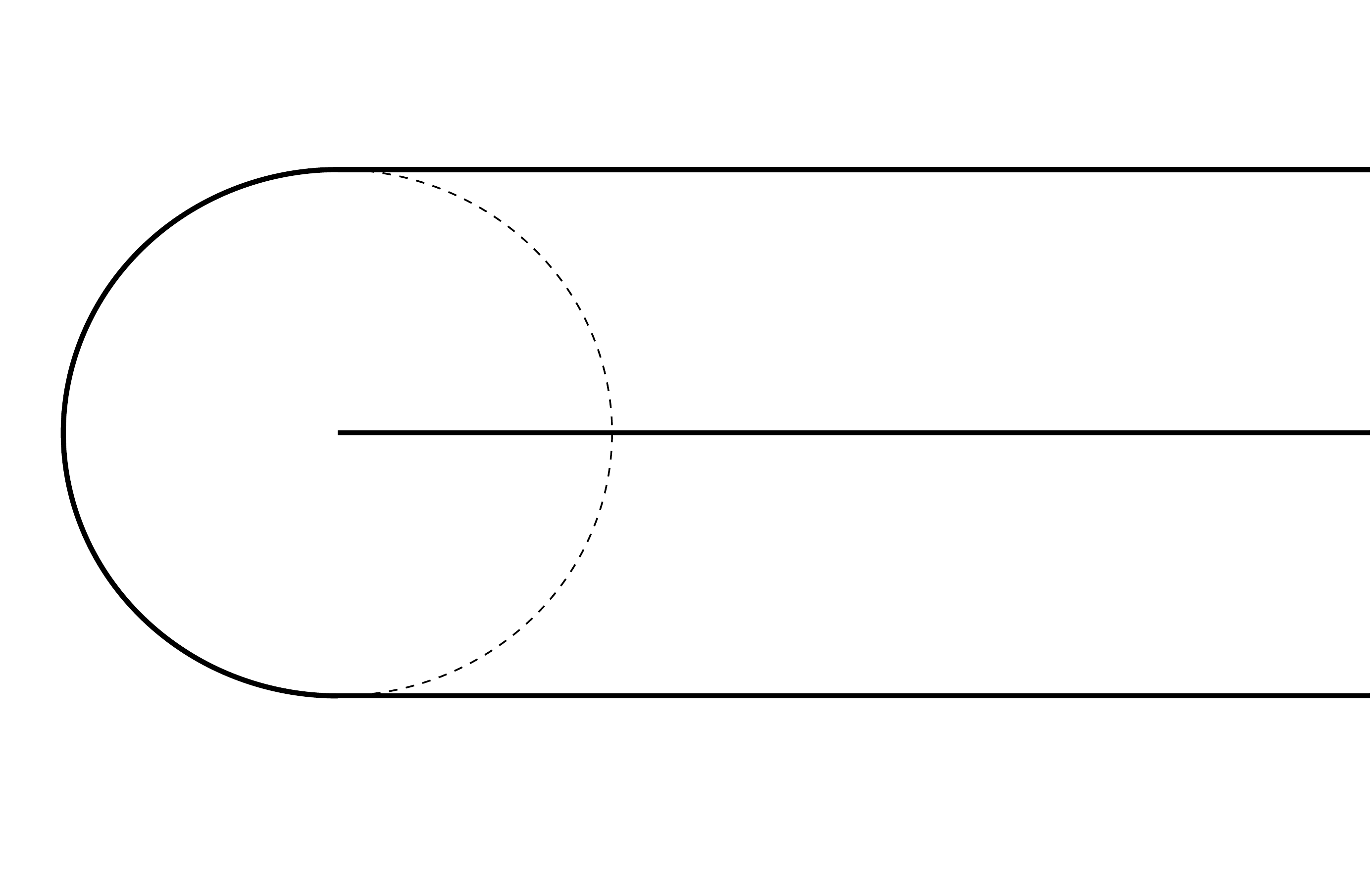}
\caption{The set $\mathbf{H}$: it can be built by taking the disk $\{(x_1,x_2)\,:\, x_1^2+x_2^2<1\}$, removing the segment $[0,1]\times\{0\}$ and then adding two parallel half-strips.}
\label{fig:hersch}
\end{figure}
Thus, it is not difficult to see that $\mathbf{H}$ is just a particular instance of the sets studied in this paper. More precisely, we observe that $\mathbf{H}$ satisfies Assumptions \ref{ass:Omega}, by taking $N=2$, $k=2$ and
\[
\Omega_0=\Big\{(x_1,x_2)\in\mathbb{R}^2\,:\, x_1^2+x_2^2<1\Big\}\setminus\Big([0,1]\times \{0\}\Big) \qquad (\mbox{\it slit disk}),
\]
\[
\mathcal{C}_1=(0,+\infty)\times(0,1),\qquad \mathcal{C}_2=(0,+\infty)\times (-1,0).
\]
It is not difficult to see that 
\[
\lambda_1(\Omega_0)=\pi^2,
\]
with a first eigenfunction given by (in polar coordinates)
\[
u(\varrho,\vartheta)=\frac{\sin(\pi\,\varrho)}{\sqrt{\pi\,\varrho}}\,\sin\left(\frac{\vartheta}{2}\right),\qquad \mbox{ for }  0<\varrho<1,\, 0<\vartheta<2\,\pi,
\]
see for example \cite[Exercise 6.8.9]{TCS}.
By strict monotonicity with respect to set inclusion, we thus have $\lambda_1(\mathbf{H})<\lambda_1(\Omega_0)=\pi^2$, while $r_{\mathbf{H}}=r_{\Omega_0}=1/2$.
Thus, by recalling the notation \eqref{rhoberto}, we have
\[
\rho(\mathbf{H})<\rho(\Omega_0)=\left(\frac{1}{2}\right)^2\,\pi^2=\frac{\pi^2}{4}.
\]
However, we will show in Section \ref{sec:7} that actually $\mathbf{H}$ {\it is not an optimal shape} for $C_{MH}$. We will prove this by a {\it singular perturbation technique}. Namely, we will see that adding a suitable (finite) number of thin tubes on the flat part of $\partial\mathbf{H}$ will decrease the quantity $\rho(\mathbf{H})$. 
In this part, we will greatly rely on the technique recently studied (in greater generality) in \cite{AO2023} by L. Abatangelo and the third author (see also \cite{FO}). We point out that, in order to apply the results of \cite{AO2023}, we will need to know that $\lambda_1(\mathbf{H})$ is attained, i.e. it admits a first eigenfunction $u_1$. It is precisely here that the results of the first part of the paper (i.e. Theorem \ref{teo:first_eigen}) will be useful.
\vskip.2cm\noindent
Let us briefly explain which is the crucial point permitting to lower the value $\rho$ by adding tubes. We first observe that if we add to $\mathbf{H}$ a small tube of width $\varepsilon\ll 1$, by calling $\mathbf{H}_\varepsilon$ the new set we would get
\[
r_{\mathbf{H}_\varepsilon}^2-r_{\mathbf{H}}^2\sim C_1\,\varepsilon^2\qquad \mbox{ and }\qquad \lambda_1(\mathbf{H})-\lambda_1({\mathbf{H}_\varepsilon})\sim C_2\,\varepsilon^2.
\]
That is, both the variation of inradius and that of $\lambda_1$ are of order $\varepsilon^2$. As one may expect, the two variations compete: the inradius {\it increases}, while $\lambda_1$ {\it decreases}. Thus, in order to decide whether the product $r_{\mathbf{H}_\varepsilon}^2 \,\lambda_1(\mathbf{H}_\varepsilon)$ decreased or not, a very precise asymptotics for $\lambda_1(\mathbf{H}_\varepsilon)$ would be needed. This seems out of reach.
On the other hand, one can observe that if we keep on adding thin tubes of the same width $\varepsilon$, the inradius is {\it insensitive} to the number $n_0$ of tubes attached. While, on the contrary, $\lambda_1$ is very much affected by these perturbations. Indeed, we can prove that each tube gives a fixed contribution of order $\varepsilon^2$, proportional to the square of the value of the normal derivative of $u_1$ at the ``junction point'' where the tube is attached. It should be noticed that, since $u_1$ has an exponential decay along the tubes, this contribution tends to be weaker and weaker, as the tube is attached further and further away from the origin.
\vskip.2cm\noindent
Incidentally, we notice that $\mathbf{H}$ can also be seen as a particular {\it curved waveguide}, i.e. it can be regarded as the tubular neighborhood (having width $1/2$) of the $C^{1,1}$ curve
\begin{equation}
\label{gamma}
\gamma(s)=\left\{\begin{array}{ll}
\left(s-1,\dfrac{1}{2}\right),& \mbox{ for } s\ge 1,\\
&\\
\left(-\dfrac{1}{2}\,\cos\left(\dfrac{\pi}{2}\,s\right),\dfrac{1}{2}\,\sin\left(\dfrac{\pi}{2}\,s\right)\right),& \mbox{ for } -1<s< 1,\\
&\\
\left(-s-1,-\dfrac{1}{2}\right),& \mbox{ for } s\le -1,\\
\end{array}
\right.
\end{equation}
see Figure \ref{fig:intorno}.
The systematic study of the spectral theory of curved waveguides has been initiated in the landmark paper \cite{ES} by Exner and \v{S}eba. In \cite[Example 4.3]{ES} this specific example is called ``bookcover''. Without any attempt of completeness, we refer to \cite{AE, EK, GJ} and \cite{KrKr} for some thorough studies on the spectral properties of these sets.
\begin{figure}
\includegraphics[scale=.2]{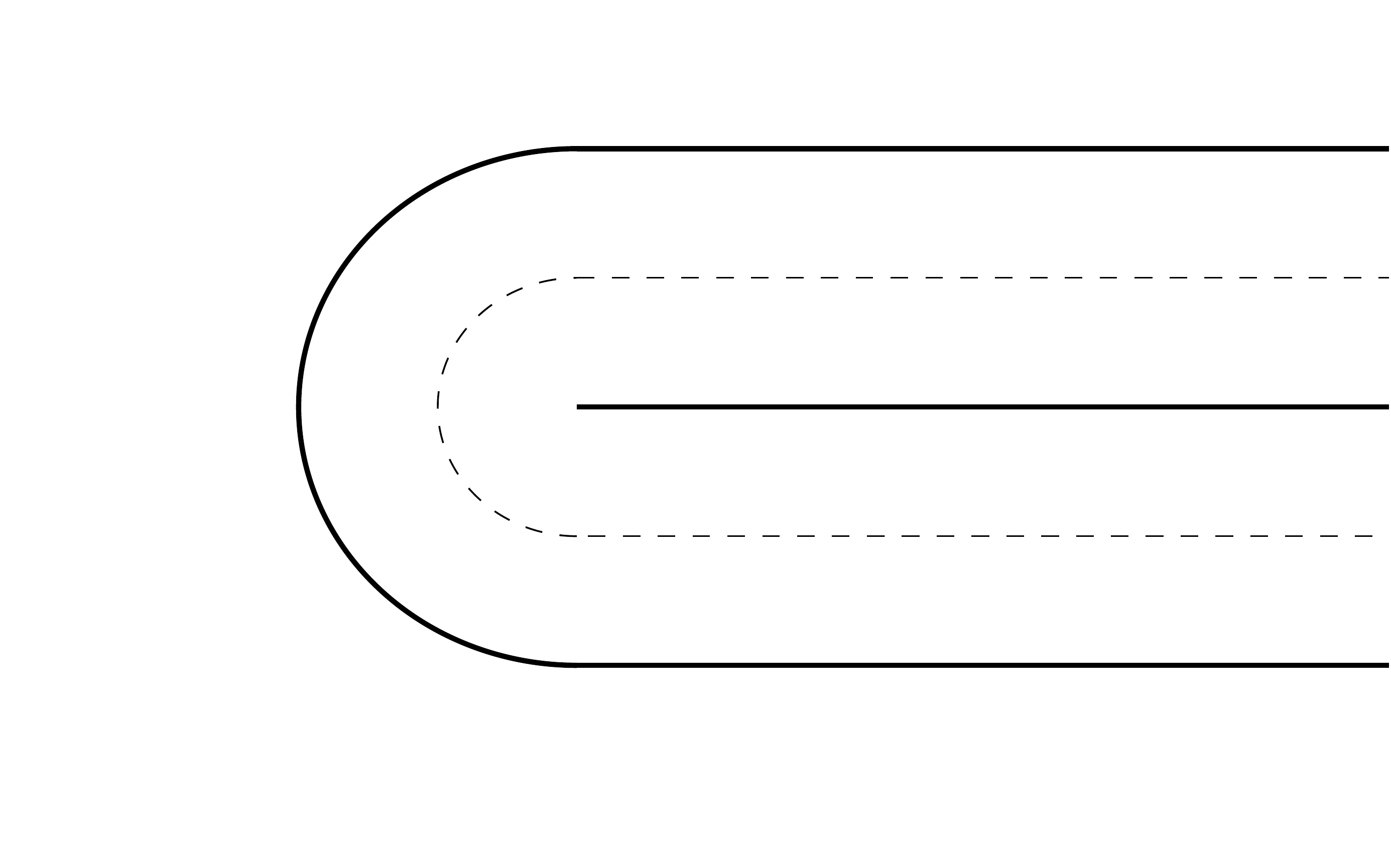}
\caption{In dashed line, the image of the curve $\gamma$ defined by \eqref{gamma}. By taking its tubular neighborhood of width $1/2$, we get the set $\mathbf{H}$.}
\label{fig:intorno}
\end{figure}
\par

\subsection{Plan of the paper} In Section \ref{sec:2} we recall some technical facts, which will be useful along the paper. Then in Section \ref{sec:3} we show that if $\lambda_1(\Omega)$ is below the critical threshold $\mathcal{E}(\Omega)$, it admits a first eigenfunction. This section contains also some examples of sets to which our existence result applies. Section \ref{sec:4} generalizes the existence result to higher eigenvalues. In the subsequent Section \ref{sec:5} we proceed to prove some properties of the eigenfunctions: the main result here is the exponential decay to $0$ at infinity.  
To complete the analysis, we discuss in Section \ref{sec:6} the loss of compactness along the tubes, i.e. we prove that every $\lambda\ge \mathcal{E}(\Omega)$ admits a constrained Palais-Smale sequence which ``disappears'' along one of the tubes $\mathcal{C}_i$. Finally, in Section \ref{sec:7} we consider Hersch's pipe and disprove its optimality for the problem of determining the sharp Makai-Hayman inequality. The paper is complemented by an appendix, which contains some technical facts taken from \cite{AO2023}.

\begin{ack}
We owe the knowledge of reference \cite{Le} to the kind courtesy of Giorgio Talenti, we wish to thank him. Part of this research has been done during a visit of R.O. to the University of Ferrara in March 2023, as well as during the conference ``{\it NonPUB23 -- Nonlocal and Nonlinear Partial Differential Equations at the University of Bologna}\,'', held in Bologna in June 2023. Organizers and hosting institutions are gratefully acknowledged.
\par
R.O. has been financially supported by the project {\it ERC VAREG -- Variational approach to the
regularity of the free boundaries} (grant agreement No. 853404) and by the INdAM-GNAMPA
Project 2022 \texttt{CUP\_E55F22000270001}.
\par
F.B. and L.B. have been financially supported by the {\it Fondo di Ateneo per la Ricerca} {\sc FAR 2020} and {\sc FAR 2021} of the University of Ferrara. 
\par
F.B. and R.O. are members of the Gruppo Nazionale per l'Analisi Matematica, la Probabilit\`a
e le loro Applicazioni (GNAMPA) of the Istituto Nazionale di Alta Matematica (INdAM). They are both supported by the INdAM-GNAMPA Project 2023 \texttt{CUP\_E53C22001930001}. 
\end{ack}

\section{Preliminaries}
\label{sec:2}

The following technical result is quite classical, it will be useful in the paper. We enclose a proof, for completeness.
\begin{lemma}
\label{lemma:sottosoluzione}
Let $\Omega\subseteq\mathbb{R}^N$ be an open set. Let $U\in W^{1,2}_0(\Omega)$ be such that
\begin{equation}
\label{equazione1}
\int_\Omega \langle \nabla U,\nabla \varphi\rangle\,dx=\lambda\,\int_\Omega U\,\varphi\,dx,\qquad \mbox{ for every } \varphi\in W^{1,2}_0(\Omega),
\end{equation}
for some $\lambda\geq 0$. Then
\begin{equation}
\label{equazione2}
\int_{\Omega} \langle \nabla |U|,\nabla \varphi\rangle\,dx\le \lambda\,\int_{\Omega} |U|\,\varphi\,dx,\qquad \mbox{ for every } \varphi\in W^{1,2}_0(\Omega) \ \mbox{ such that } \varphi\ge 0.
\end{equation}
\end{lemma}
\begin{proof}
For every $\varepsilon>0$, we consider the $C^2$ convex function
\[
f_\varepsilon(t)=(\varepsilon^2+t^2)^\frac{1}{2}-\varepsilon,\qquad \mbox{ for } t\in\mathbb{R}.
\]
By the Chain Rule, we have $f_\varepsilon(U)\in W^{1,2}_0(\Omega)$ and $\nabla f_\varepsilon(U)=f'_\varepsilon(U)\,\nabla U$. We take a non-negative $\eta\in C^\infty_0(\Omega)$ and use the test function
\[
\varphi=\eta\,f_\varepsilon'(U),
\]
in \eqref{equazione1}. This gives
\[
\int_\Omega \langle \nabla f_\varepsilon\circ U,\nabla\eta\rangle\,dx+\int_\Omega f_\varepsilon''(U)\,|\nabla U|^2\,\eta\,dx= \lambda\,\int_\Omega U\,\eta\,f'_\varepsilon(U)\,dx.
\]
By using the convexity of $f_\varepsilon$ and the fact that
\[
t\,f_\varepsilon'(t)\le |t|,\qquad \mbox{ for every } t\in\mathbb{R},
\]
we obtain
\begin{equation}
\label{equazione3}
\int_\Omega \langle \nabla f_\varepsilon\circ U,\nabla\eta\rangle\,dx\le \lambda\,\int_\Omega |U|\,\eta\,dx.
\end{equation}
By using the Dominated Convergence Theorem and the form of $f_\varepsilon$, it is easily seen that
\[
\lim_{\varepsilon\to 0^+} \big\| f_\varepsilon\circ U-|U|\big\|_{W^{1,2}(\Omega)}=0.
\] 
Thus we can pass to the limit in \eqref{equazione3}. Finally, by a density argument, we can enlarge the class of test functions to $\eta\in W^{1,2}_0(\Omega)$, with $\eta\ge 0$, and conclude the proof.
\end{proof}
We recall that for a general open set $\Omega\subseteq\mathbb{R}^N$, we have defined for every $j\in\mathbb{N}\setminus\{0\}$
\begin{equation*}
	\lambda_j(\Omega):=\inf\left\{ \max_{\varphi\in F\cap \mathcal{S}_2(\Omega)} \int_\Omega|\nabla \varphi|^2\,dx\,:\, F \subseteq W^{1,2}_0(\Omega) \mbox{ subspace with } \dim F=j \right\},
\end{equation*}
where
\[
\mathcal{S}_2(\Omega)=\Big\{\varphi\in W^{1,2}_0(\Omega)\, :\, \|\varphi\|_{L^2(\Omega)}=1\Big\}.
\]
Thanks to its definition, it is easily seen that for every pair of open sets $\Omega',\Omega\subseteq\mathbb{R}^N$, we have
\begin{equation}
\label{monotoni}
\Omega'\subseteq\Omega\qquad \Longrightarrow \qquad \lambda_j(\Omega)\le \lambda_j(\Omega'),\ \mbox{ for every } j\in\mathbb{N}\setminus\{0\}.
\end{equation}
The next two results are apparently well-known, but we draw the reader's attention to the fact that we do not take any assumption on the open set. In particular, $\lambda_j$ is not necessarily an eigenvalue.
\begin{lemma}\label{lm:limlamk}
Let $\Omega\subseteq\mathbb{R}^N$ be an open set and let $\{\Omega_n\}_{n\in\mathbb{N}}$ be a sequence of open sets such that
\[
\Omega_n\subseteq \Omega_{n+1}\subseteq \Omega\qquad \mbox{ and }\qquad \bigcup_{n\in\mathbb{N}} \Omega_n=\Omega.
\]
Then we have
\[
\lim_{n\to\infty} \lambda_j(\Omega_n)=\lambda_j(\Omega),\qquad \mbox{ for every } j\in\mathbb{N}\setminus\{0\}.
\]
\end{lemma}
\begin{proof}
By \eqref{monotoni}, we already know that the limit of $\lambda_j(\Omega)$ exists and is such that
\[
\lim_{n\to\infty} \lambda_j(\Omega_n)\ge \lambda_j(\Omega).
\]
In order to prove the reverse inequality, we take $F\subseteq W^{1,2}_0(\Omega)$ a vector subspace with dimension $j$. By definition, this means that there exists $j$ linearly independent functions $u_1,\dots,u_j\in W^{1,2}_0(\Omega)$ such that
\[
F=\left\{\sum_{i=1}^j \alpha_i\,u_i\, :\, \alpha_1,\dots,\alpha_j\in\mathbb{R} \right\}.
\]
For every $u_i$ with $i=1,\dots,j$, there exists a sequence $\{u_i^m\}_{m\in\mathbb{N}}\subseteq C^\infty_0(\Omega)$ such that
\[
\lim_{m\to\infty} \|u_i^m-u_i\|_{W^{1,2}(\Omega)}=0,\qquad \mbox{ for every } i=1,\dots,j.
\]
We then define
\[
F_m=\left\{\sum_{i=1}^j \alpha_i\,u^m_i\, :\, \alpha_1,\dots,\alpha_j\in\mathbb{R} \right\},
\]
and observe that this is a $j-$dimensional vector subspace of $C^\infty_0(\Omega)$, for $m$ large enough. Since the sequence $\{\Omega_n\}_{n\in\mathbb{N}}$ is exhausting $\Omega$, we have that $F_m$ is actually a $j-$dimensional vector subspace of $C^\infty_0(\Omega_n)$, for $n$ large enough (depending on $m$). Thus we get
\[
\lim_{n\to\infty}\lambda_j(\Omega_n)\le \max_{u \in F_m\cap \mathcal{S}_2(\Omega)} \int_\Omega |\nabla u|^2\,dx.
\]
By using the construction of $F_m$, it is easily seen that
\[
\lim_{m\to\infty}\max_{u \in F_m\cap \mathcal{S}_2(\Omega)} \int_\Omega |\nabla u|^2\,dx=\max_{u \in F\cap \mathcal{S}_2(\Omega)} \int_\Omega |\nabla u|^2\,dx.
\]
Finally, by arbitrariness of $F$, the last two equations in display and the definition of $\lambda_j(\Omega)$ give the desired conclusion.
\end{proof}
\begin{lemma}
\label{lemma:crescono}
Let $\Omega\subseteq\mathbb{R}^N$ be an open set, then we have
\[
\lambda_j(\Omega)\le \lambda_{j+1}(\Omega),\qquad \mbox{ for every } j\in\mathbb{N}\setminus\{0\}.
\]
\end{lemma}
\begin{proof}
		Let $F\subseteq W^{1,2}_0(\Omega)$ be a $(j+1)-$dimensional vector subspace and let $\{u_1,\dots,u_j,u_{j+1}\}$ be a basis. We define $\widetilde{F}$ the $j-$dimensional subspace spanned by $\{u_1,\dots,u_j\}$. 
		Then, $\widetilde{F}\subseteq F$ and we have that
		\[
	\lambda_j(\Omega)\le		\max_{u\in \widetilde{F}\cap\mathcal{S}_2(\Omega)}\int_\Omega|\nabla u|^2\,dx\leq \max_{u\in F\cap\mathcal{S}_2(\Omega)}\int_\Omega|\nabla u|^2\,dx.
		\]
By taking the infimum with respect to $(j+1)-$dimensional subspaces $F$, we conclude.
\end{proof}

\section{The first eigenvalue}
\label{sec:3}

\subsection{Set-up}
\label{sec:3.1}
We will use the same notations of Assumptions \ref{ass:Omega}.
For any $r>0$, we set 
\[
Q_r^i:=\textbf{R}_i\big((-r,r)^N\big)+\mathbf{t}_i,\qquad i=1,\dots,k,
\]
and
\[
\Omega^i_r:=\Omega_0\cup (Q_r^i\cap \mathcal{C}_i),\qquad \mbox{ for } i=1,\dots,k.
\]
Accordingly, we also set
\begin{equation}
\label{Omegar}
\Omega_r=\bigcup_{i=1}^k\Omega_r^i.
\end{equation}
We then take $r_0=r_0(\Omega)>0$ to be large enough such that $\Omega\setminus\Omega_r$ is a disjoint union of $k$ cylindrical sets, for any $r\geq r_0$. More precisely, we have
\begin{equation}\label{eq:r_0}
	\Omega\setminus\Omega_r=\bigcup_{i=1}^k (\mathcal{C}_i\setminus Q_r^i),\qquad \mbox{ for }r\ge r_0.
\end{equation}
The existence of such $r_0$ is guaranteed by the construction of $\Omega$, see Assumptions \ref{ass:Omega}. 
\par
This crucial property entails that 
a Poincar\'e inequality holds for functions of $W^{1,2}_0(\Omega)$, even when restricted to $\mathcal{C}_i\setminus Q^i_r$. Namely, we have the following

\begin{lemma}[Poincar\'e inequality at infinity]\label{lemma:poinc}
Let $\Omega\subseteq\mathbb{R}^N$ be an open set satistying Assumptions \ref{ass:Omega}.
With the notations above, we have	
	\begin{equation*}
		\lambda_1(E_i)\,\int_{\mathcal{C}_i\setminus Q^i_r}|\varphi|^2\,dx\leq \int_{\mathcal{C}_i\setminus Q^i_r}|\nabla \varphi|^2\,dx,
	\end{equation*}
for every $i=1,\dots,k$, $r\geq r_0$ and every $\varphi\in W^{1,2}_0(\Omega)$.
\end{lemma}
\begin{proof}
We will use the notation $x=(x',x_N)\in\mathbb{R}^{N-1}\times \mathbb{R}$. By a simple change of variable, we have that, for any $\varphi\in C^\infty_0(\Omega)$ and any $i=1,\dots,k$, there holds
	\[
	\int_{\mathcal{C}_i\setminus Q_r^i}|\varphi|^2\,dx=\int_r^{+\infty}\int_{E_i}|\varphi(\textbf{R}_i(x',x_N)+\mathbf{t}_i)|^2\,dx'\,dx_N.
	\]
For every fixed $x_N$, we now use the $(N-1)-$dimensional Poincar\'e inequality for the function 
\[
x'\mapsto \varphi(\textbf{R}_i(x',x_N)+\mathbf{t}_i),
\] 
which is compactly supported in $E_i$. This yields 
\[
\int_{\mathcal{C}_i\setminus Q^i_r}|\varphi|^2\,dx\leq \frac{1}{\lambda_1(E_i)}\,\int_r^{+\infty}\int_{E_i}|\nabla'\varphi(\textbf{R}_i(x',x_N)+\mathbf{t}_i)|^2\,dx'\,dx_N,
\]
where we denoted by $\nabla'$ the gradient with respect to $x'$. By observing that 
\[
\int_r^{+\infty}\int_{E_i}|\nabla'\varphi(\textbf{R}_i(x',x_N)+\mathbf{t}_i)|^2\,dx'\,dx_N\le \int_{\mathcal{C}_{i}\setminus Q_r^i}|\nabla \varphi|^2\,dx,
\]
we now easily get the desired Poincar\'e inequality, for functions in $C^\infty_0(\Omega)$. Finally, by density of $C^\infty_0(\Omega)$ in $W^{1,2}_0(\Omega)$, we conclude the proof.
\end{proof}

\begin{definition}[Tubular cut-off functions]\label{def:cutoff}
	We denote by $\eta:\mathbb{R}\to [0,+\infty)$ a $C^\infty$ function such that 
	\[
	0\leq \eta\leq 1,\qquad |\eta'|\leq 2,\qquad \eta(t)=\begin{cases}
		0, &\text{for }t\leq 0, \\
		1, &\text{for }t\geq 1.
	\end{cases}
\]
For every $i=1,\dots,k$, we set 
\[
\eta_{E_i}(x',x_N)=\left\{\begin{array}{rl}
\eta(x_N),& \mbox{ if } x'\in \overline{E_i},\\
0,& \mbox{ otherwise},
\end{array}
\right.
\] 
and for a pair $0<r<R$, we also define
\[
\eta_{r,R}^{(i)}(x)=\eta_{E_i}\left(\frac{\mathbf{R}_i^{-1}(x-\mathbf{t}_i)-r\,\mathbf{e}_N}{R-r}\right).
\]
It is crucial to observe that, by construction, we have 
\[
\eta^{(i)}_{r,R}(x)\equiv 0,\ \mbox{ in } \mathcal{C}_i\cap Q_r^i,\qquad \eta^{(i)}_{r,R}(x)\equiv 1,\  \mbox{ in } \mathcal{C}_i\setminus Q^i_{R},\qquad \left|\nabla \eta^{(i)}_{r,R}(x)\right|\le \frac{2}{R-r},\ \mbox{ in } \mathcal{C}_i.
\]
Here $\mathbf{R}_i$ and $\mathbf{t}_i$ are still as in Assumptions \ref{ass:Omega}.
In particular, this acts as a ``cut-off at infinity'' along the cylindrical set $\mathcal{C}_i$.  In the particular case $R=r+1$, we will simply use the notation
\begin{equation*}
	\eta_r^{(i)}(x):=\eta_{E_i}\big(\mathbf{R}_i^{-1}(x-\mathbf{t}_i)-r\,\mathbf{e}_N\big).
\end{equation*}
\end{definition}

\subsection{Existence of a first eigenfunction}

\begin{theorem}
\label{teo:first_eigen}
Let $\Omega\subseteq\mathbb{R}^N$ be an open set satisfying Assumptions \ref{ass:Omega}. If 
	\begin{equation}
	\label{energybelow}
		\lambda_1(\Omega)<\mathcal{E}(\Omega):=\min\Big\{\lambda_1(E_1),\dots,\lambda_1(E_k)\Big\}.
	\end{equation}
then $\lambda_1(\Omega)$ is actually a minimum, i.e. there exists a first eigenfunction $u_1\in W^{1,2}_0(\Omega)$ and $u_1\geq 0$. Moreover, if $\Omega$ is connected, then every other first eigenfunction is proportional to $u_1$.
\end{theorem}
\begin{proof}
Let $n\geq r_0$. With the notation of \eqref{Omegar}, we consider $\Omega_n\subseteq\mathbb{R}^N$. Since the latter is an open {\it bounded} set, the embedding $W^{1,2}_0(\Omega_n)\hookrightarrow L^2(\Omega_n)$ is compact and the quantity 
	\begin{equation*}
		\lambda_1(\Omega_n)=\inf_{\varphi\in \mathcal{S}_2(\Omega_n)} \int_{\Omega_n} |\nabla \varphi|^2\,dx,
	\end{equation*}
admits a nonnegative minimizer $u_{1}^n\in \mathcal{S}_2(\Omega_n)$, as recalled in the Introduction. 
We also observe that 
\begin{equation}
\label{eq:first_eigen_1}
\lim_{n\to\infty}\lambda_1(\Omega_n)=\lambda_1(\Omega),
\end{equation}
thanks to Lemma \ref{lm:limlamk}. We will prove existence of a first eigenfunction for $\Omega$ by using the Direct Method in the Calculus of Variations: as a minimizing sequence, we will take $\{u_1^n\}_{n\in\mathbb{N}}$, with the previous notation. Indeed, in light of \eqref{eq:first_eigen_1}, this is a minimizing sequence.
\par
In what follows, each function $u_1^n$ is considered to be extended by $0$ outside $\Omega_n$. We first observe that this sequence is bounded in $W^{1,2}_0(\Omega)$.
Thus, there exists $u_1\in W^{1,2}_0(\Omega)$ such that $u_1^n$ converges weakly to $u_1$ in $W^{1,2}_0(\Omega)$, up to a subsequence. Observe that the limit function still belongs to $W^{1,2}_0(\Omega)$, because the latter is a weakly closed space. Moreover, by lower semicontiuity, we have
\[
\int_\Omega |\nabla u_1|^2\,dx\le \liminf_{n\to\infty} \int_\Omega |\nabla u_1^n|^2\,dx=\lambda_1(\Omega). 
\]
In order to conclude, it is sufficient to upgrade this convergence to the strong one in $L^2(\mathbb{R}^N)$, so to assure that $u_1\in\mathcal{S}_2(\Omega)$.
To this aim, we appeal to the classical Riesz-Fr\'echet-Kolmogorov Theorem. Since we work with a sequence of functions in $W^{1,2}_0(\Omega)$, which is bounded in the norm of $W^{1,2}(\mathbb{R}^N)$, 
we just need to verify that the sequence ``does not lose mass at infinity''. Namely, we need to prove that there exists $n_0\in\mathbb{N}$ such that, for every $\varepsilon>0$, there exists $R_\varepsilon>0$ such that
	\begin{equation}\label{eq:RFK}
		\int_{\mathbb{R}^N\setminus \Omega_{R_\varepsilon}}|u_1^n|^2\,dx<\varepsilon,\qquad\text{ for all }n\geq n_0.
	\end{equation}
By minimality, we have that	
	\begin{equation}\label{eq:first_eigen_2}
		\int_{\mathbb{R}^N} 
		\langle\nabla u_1^n,\nabla \varphi\rangle\,dx= \lambda_1(\Omega_n)\,\int_{\mathbb{R}^N}u_1^n\,\varphi\,dx,\qquad\mbox{ for every }\varphi\in W^{1,2}_0(\Omega_n).
	\end{equation}
	We now fix $i=1,\dots,k$ and test \eqref{eq:first_eigen_2} with $u_1^n\,(\eta_r^{(i)})^2\in W^{1,2}_0(\Omega_n)$, where $\eta^{(i)}_r$ is as\footnote{This is a feasible test function, since for every $\varphi\in W^{1,2}_0(\Omega_n)$ and $\eta\in C^1(\overline{\Omega_n})$, we have $\varphi\,\eta\in W^{1,2}_0(\Omega_n)$, as well.} in Definition \ref{def:cutoff} and $r\ge r_0$. This gives
	\begin{equation}\label{eq:first_eigen_3}
		\int_{\mathbb{R}^N}|\nabla u_1^n|^2\,\left|\eta_r^{(i)}\right|^2\,dx+2\,\int_{\mathbb{R}^N}u_1^n\,\eta^{(i)}_r\,\langle\nabla u_1^n,\nabla \eta_r^{(i)}\rangle\,dx=\lambda_1(\Omega_n)\,\int_{\mathbb{R}^N}|u_1^n|^2\,\left|\eta_r^{(i)}\right|^2\,dx.
	\end{equation}
	By Young's inequality and taking into account the properties of $\eta_r^{(i)}$, for $0<\delta_n<1$ we have that
	\[
	\begin{split}
		\int_{\mathbb{R}^N}|\nabla u_1^n|^2\,\left|\eta_r^{(i)}\right|^2\,dx&+2\,\int_{\mathbb{R}^N}u_1^n\,\eta^{(i)}_r\,\langle\nabla u_1^n,\nabla \eta_r^{(i)}\rangle\,dx\\
		&\geq \int_{\mathbb{R}^N}|\nabla u_1^n|^2\,\left|\eta_r^{(i)}\right|^2\,dx-\delta_n\,\int_{\mathbb{R}^N}|\nabla u_1^n|^2\,\left|\eta_r^{(i)}\right|^2\,dx-\frac{1}{\delta_n}\,\int_{\mathbb{R}^N} |u_1^n|^2\,|\nabla \eta_r^{(i)}|^2\,dx \\
		&\geq (1-\delta_n)\,\int_{\mathcal{C}_i\setminus \Omega^i_{r+1}}|\nabla u_1^n|^2\,dx-\frac{4}{\delta_n}\int_{\Omega^i_{r+1}\setminus\Omega^i_r} |u_1^n|^2\,dx.
	\end{split}
	\]
	By spending this information in \eqref{eq:first_eigen_3} and using Lemma \ref{lemma:poinc}, we obtain that
	\[
	\mathcal{E}(\Omega)\,(1-\delta_n)\,\int_{\mathcal{C}_i\setminus \Omega^i_{r+1}}|u_1^n|^2\,dx\le \frac{4}{\delta_n}\int_{\Omega^i_{r+1}\setminus\Omega^i_r} |u_1^n|^2\,dx+\lambda_1(\Omega_n)\,\int_{\mathcal{C}_i\setminus\Omega^i_r}|u_1^n|^2\,dx.
	\]
We can decompose	the last term as follows
\[
\lambda_1(\Omega_n)\,\int_{\mathcal{C}_i\setminus\Omega^i_r}|u_1^n|^2\,dx=\lambda_1(\Omega_n)\,\int_{\mathcal{C}_i\setminus\Omega^i_{r+1}}|u_1^n|^2\,dx+\lambda_1(\Omega_n)\,\int_{\Omega^i_{r+1}\setminus\Omega^i_r}|u_1^n|^2\,dx.
\]
This in turn gives
	\begin{equation}
	\label{scappotta}
		\Big((1-\delta_n)\,\mathcal{E}(\Omega)-\lambda_1(\Omega_n)\Big)\,\int_{\mathcal{C}_i\setminus \Omega^i_{r+1}} |u_1^n|^2\,dx\leq \left(\frac{4}{\delta_n}+\lambda_1(\Omega_n)\right)\,\int_{\Omega^i_{r+1}\setminus\Omega^i_r} |u_1^n|^2\,dx.
	\end{equation}
	In light  of \eqref{eq:first_eigen_1} and of the assumption $\lambda_1(\Omega)<\mathcal{E}(\Omega)$, there exists $n_0\ge r_0$ such that
\[
\lambda_1(\Omega_n)\le \frac{\mathcal{E}(\Omega)+\lambda_1(\Omega)}{2},\qquad \mbox{ for every } n\ge n_0.
\]
Thus, we also get
\[
\mathcal{E}(\Omega)-\lambda_1(\Omega_n)\ge \frac{\mathcal{E}(\Omega)-\lambda_1(\Omega)}{2}>0,\qquad \mbox{ for every } n\ge n_0
\]
This entails that for every $n\ge n_0$ the following choice is feasible
\[
\delta_n=\frac{\mathcal{E}(\Omega)-\lambda_1(\Omega_n)}{2\,\mathcal{E}(\Omega)}.
\]
From \eqref{scappotta} we then obtain
\begin{equation}
	\label{riscappotta}
	\frac{\mathcal{E}(\Omega)-\lambda_1(\Omega_n)}{2}\,\int_{\mathcal{C}_i\setminus \Omega^i_{r+1}} |u_1^n|^2\,dx\leq \left(\frac{8\,\mathcal{E}(\Omega)}{\mathcal{E}(\Omega)-\lambda_1(\Omega_n)}+\lambda_1(\Omega_n)\right)\,\int_{\Omega^i_{r+1}\setminus\Omega^i_r} |u_1^n|^2\,dx.
\end{equation}
for every $n\ge n_0$. Observe that for $n\ge n_0$ we also have
\[
\lambda_1(\Omega_n)\le \frac{\mathcal{E}(\Omega)+\lambda_1(\Omega)}{2}<\mathcal{E}(\Omega),
\] 
thus from \eqref{riscappotta} we can get 
\[
\int_{\mathcal{C}_i\setminus \Omega^i_{r+1}} |u_1^n|^2\,dx\leq C_\Omega\,\int_{\Omega^i_{r+1}\setminus\Omega^i_r} |u_1^n|^2\,dx,\qquad \mbox{ for } n\ge n_0 \mbox{ and } i=1,\dots,k,
\]
where $C_\Omega>0$ is the constant given by
\[
C_\Omega:=\frac{4}{\mathcal{E}(\Omega)-\lambda_1(\Omega)}\,\left(\frac{16\,\mathcal{E}(\Omega)}{\mathcal{E}(\Omega)-\lambda_1(\Omega)}+\mathcal{E}(\Omega)\right).
\]
We can sum over $i$ such an estimate, so to get 
\begin{equation}
\label{predecay}
\int_{\mathbb{R}^N\setminus \Omega_{r+1}} |u_1^n|^2\,dx=\sum_{i=1}^k\int_{\mathcal{C}_i\setminus \Omega^i_{r+1}} |u_1^n|^2\,dx\le C_\Omega\,\int_{\Omega_{r+1}\setminus\Omega_r} |u_1^n|^2\,dx,
\end{equation}
which holds for $r\ge r_0$ and $n\ge n_0$. It is not difficult to see that this estimate gives the desired uniform decay at infinity: indeed, if we set
\[
A_n(r)=\int_{\mathbb{R}^N\setminus \Omega_{r}} |u_1^n|^2\,dx,
\]
the previous estimate can be recast into
\[
A_n(r+1)\le C_\Omega\,(A_n(r)-A_n(r+1))\qquad \mbox{ that is }\qquad A_n(r+1)\le \frac{C_\Omega}{C_\Omega+1}\,A_n(r).
\]
In particular, by using that $r\mapsto A_n(r)$ is non-increasing, for every $R\ge r_0$ we can obtain\footnote{For every $t\in\mathbb{R}$, we denote its {\it integer part} by 
\[
\big\lfloor t\big\rfloor=\max\Big\{n\in\mathbb{Z}\, :\, t\ge n\Big\}.
\]
}
\[
A_n(R+1)\le A_n(r_0+\big\lfloor R-r_0\big\rfloor+1)\le \left(\frac{C_\Omega}{C_\Omega+1}\right)^{\big\lfloor R-r_0\big\rfloor+1}\,A_n(r_0)\le \left(\frac{C_\Omega}{C_\Omega+1}\right)^{R-r_0}\, A_n(r_0).
\]  
By using that
\[
A_n(r_0)\le \int_{\mathbb{R}^N} |u_1^n|^2\,dx=1,
\]
we finally get the desired decay property \eqref{eq:RFK}. Therefore, up to a subsequence, we have that $u_1^n$ strongly converges to $u_1$ in $L^2(\Omega)$, as $n$ goes to $\infty$. Moreover, $u_1\geq 0$ since $u_1^n\geq 0$ for all $n\in\mathbb{N}$. 
\par
The last part of the statement is a classical fact.
This completes the proof.
\end{proof}
\begin{remark}
If we admit the cylindrical sets $\mathcal{C}_i$ to have an overlapping with infinite volume, i.e. if condition (A3) is violated, the previous result may fail to be true. As a simple example, take $N=2$, $k=2$ and 
\[
\Omega_0=(0,2)\times(0,2),\qquad \mathcal{C}_1=(0,2)\times(0,+\infty),\qquad \mathcal{C}_2=(1,3)\times(0,+\infty).
\] 
In this case, we have 
\[
\Omega=\Omega_0\cup \mathcal{C}_1\cup \mathcal{C}_2=(0,3)\times(0,+\infty),
\]
and 
\[
\lambda_1(\Omega)=\frac{\pi^2}{9}<\min\Big\{\lambda_1(\mathcal{C}_1),\lambda_1(\mathcal{C}_2)\Big\}=\frac{\pi^2}{4}.
\]
However, it is well-know that $\lambda_1(\Omega)$ is not attained, in this case.
\end{remark}

\subsection{Examples}

\begin{example}[Massive core]
\label{exa:massive}
Let us suppose that $\Omega\subseteq\mathbb{R}^N$ satisfies Assumptions \ref{ass:Omega}, with its core $\Omega_0$ having the following property: 
\[
\mbox{ if }\mathcal{E}(\Omega)=\lambda_1(E_i), \mbox{ then }\quad \Omega_0\cap \mathcal{C}_i\not=\emptyset\quad \mbox{ and }\quad \lambda_1(\Omega_0)\le \lambda_1(\mathcal{C}_i).
\]
In particular, we have $\Omega_0\subseteq \Omega_0\cup \mathcal{C}_i\subseteq \Omega$ and $\Omega_0\cup \mathcal{C}_i$ is a connected open set such that 
\[
|(\Omega_0\cup \mathcal{C}_i)\setminus \Omega_0|>0.
\]
Moreover, $\Omega_0$ is bounded and thus $\lambda_1(\Omega_0)$ is actually an eigenvalue. These facts imply that 
\[
\lambda_1(\Omega_0\cup \mathcal{C}_i)<\lambda_1(\Omega_0).
\]
This in turn gives
\[
\lambda_1(\Omega)<\lambda_1(\Omega_0)\le \lambda_1(\mathcal{C}_i)=\mathcal{E}(\Omega).
\]
Thus, we can apply Theorem \ref{teo:first_eigen}.
\end{example}

\begin{example}[Infinite cross]
\label{exa:cross}
We take the open set 
\[
\Omega=\Big(\mathbb{R}\times (-1,1)\Big)\times \Big((-1,1)\times \mathbb{R}\Big).
\]
Observe that such a set satisfies Assumptions \ref{ass:Omega}: indeed, it can be also written as follows
\[
\Omega=Q\cup \bigcup_{i=1}^4 \mathcal{C}_i,
\]
where $Q$ is the open square with vertices $(0,\pm 2)$ and $(\pm 2,0)$ and 
\[
\mathcal{C}_1=(-1,+\infty)\times (-1,1),\qquad \mathcal{C}_1=(-\infty,1)\times (-1,1),
\]
\[
\mathcal{C}_3=(-1,1)\times (1,+\infty),\qquad\mathcal{C}_4=(-1,1)\times (-\infty,1).
\]
We observe that 
\[
\mathcal{E}(\Omega)=\lambda_1(\mathcal{C}_1)=\dots=\lambda_1(\mathcal{C}_4)=\frac{\pi^2}{4},
\]
and that (see \cite[Chapter 1]{Hen})
\[
\lambda_1(Q)=\frac{\pi^2}{(2\,\sqrt{2})^2}+\frac{\pi^2}{(2\,\sqrt{2})^2}=\mathcal{E}(\Omega).
\] 
Thus, this set is a particular case of Example \ref{exa:massive} and we can apply Theorem \ref{teo:first_eigen}. A study of the spectral properties of this set can be also found in \cite{ABGM, Na} and \cite{SRW}. 
\par
We point out that this set could also be realized simply as the union of the four cylindrical sets $\mathcal{C}_1,\dots,\mathcal{C}_4$, i.e. we could think that the core $\Omega_0=\emptyset$. This shows that Theorem \ref{teo:first_eigen} may cover cases where the core is empty and the ``compactness'' is created by some non-trivial intersections of the ``tubes''.
\begin{figure}
\includegraphics[scale=.2]{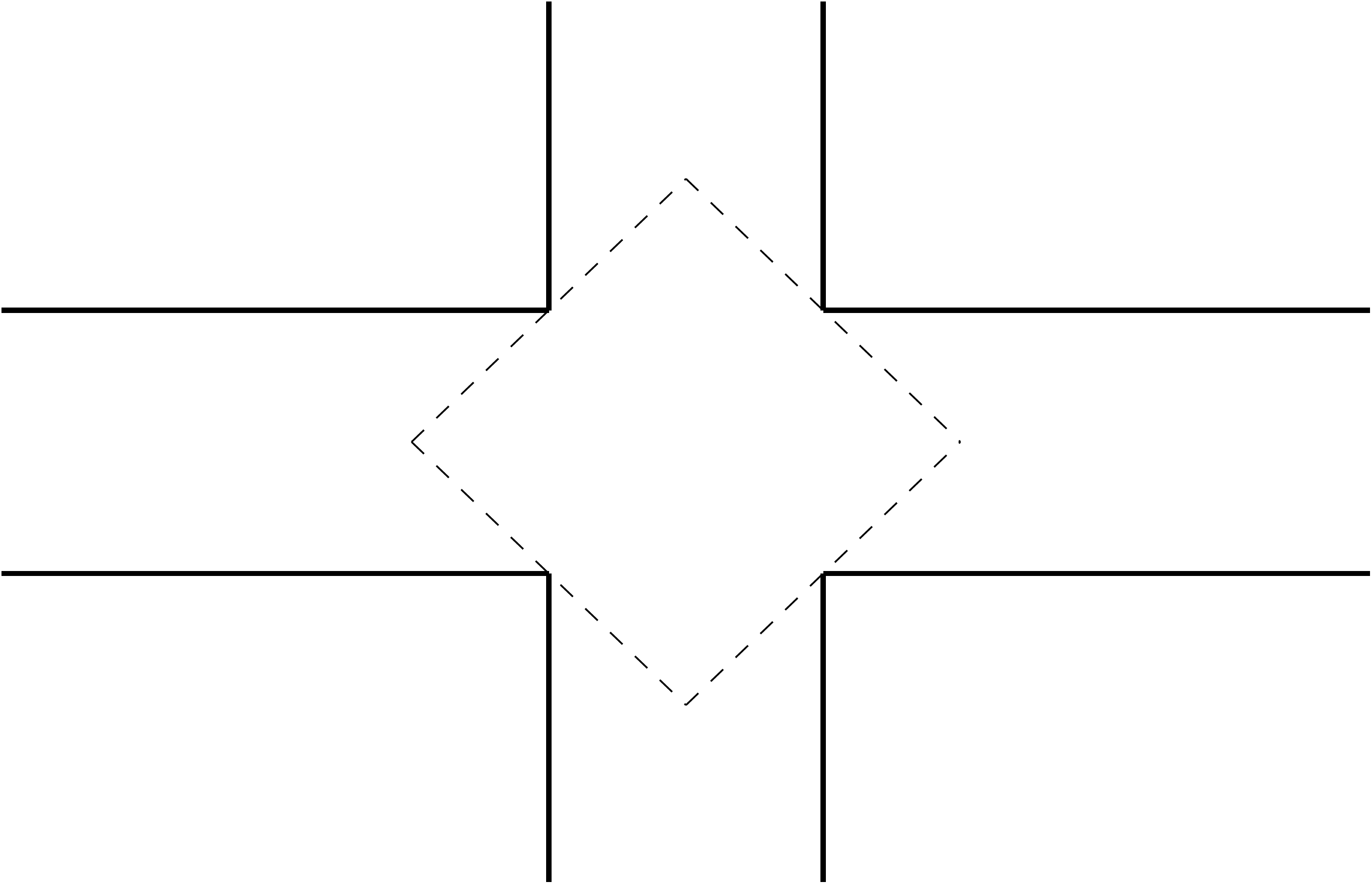}
\caption{The set of Example \ref{exa:cross}. In dashed line, the square $Q$ such that $\lambda_1(Q)\le \mathcal{E}(\Omega)$.}
\end{figure}
\end{example}

\begin{example}[The broken strip]
We fix $0<\vartheta<\pi/2$ and take the open set $\Omega$ defined as follows
\[
\Omega=\left\{(x_1,x_2)\in\mathbb{R}^2\,:\, \left|\frac{x_2}{\tan\vartheta}\right|-\frac{1}{\sin\vartheta}<x_1<\left|\frac{x_2}{\tan\vartheta}\right|\right\},
\]
see Figure \ref{fig:guida_angolare}. This set satisfies Assumptions \ref{ass:Omega}, by taking for example
\[
\Omega_0=T:=\Big\{(x_1,x_2)\in\Omega\,:\, x_1<0\Big\} \qquad (\mbox{\it isosceles triangle}),
\]
and the two half-strips
\[
\mathcal{C}_1=\Big\{(x_1,x_2)\in\Omega\, :\, x_1\,\cos\vartheta+x_2\,\sin\vartheta>0\Big\},
\]
\[
\mathcal{C}_2=\Big\{(x_1,x_2)\in\Omega\, :\, x_1\,\cos\vartheta-x_2\,\sin\vartheta>0\Big\}.
\]
\label{exa:strip}
\begin{figure}
\includegraphics[scale=.2]{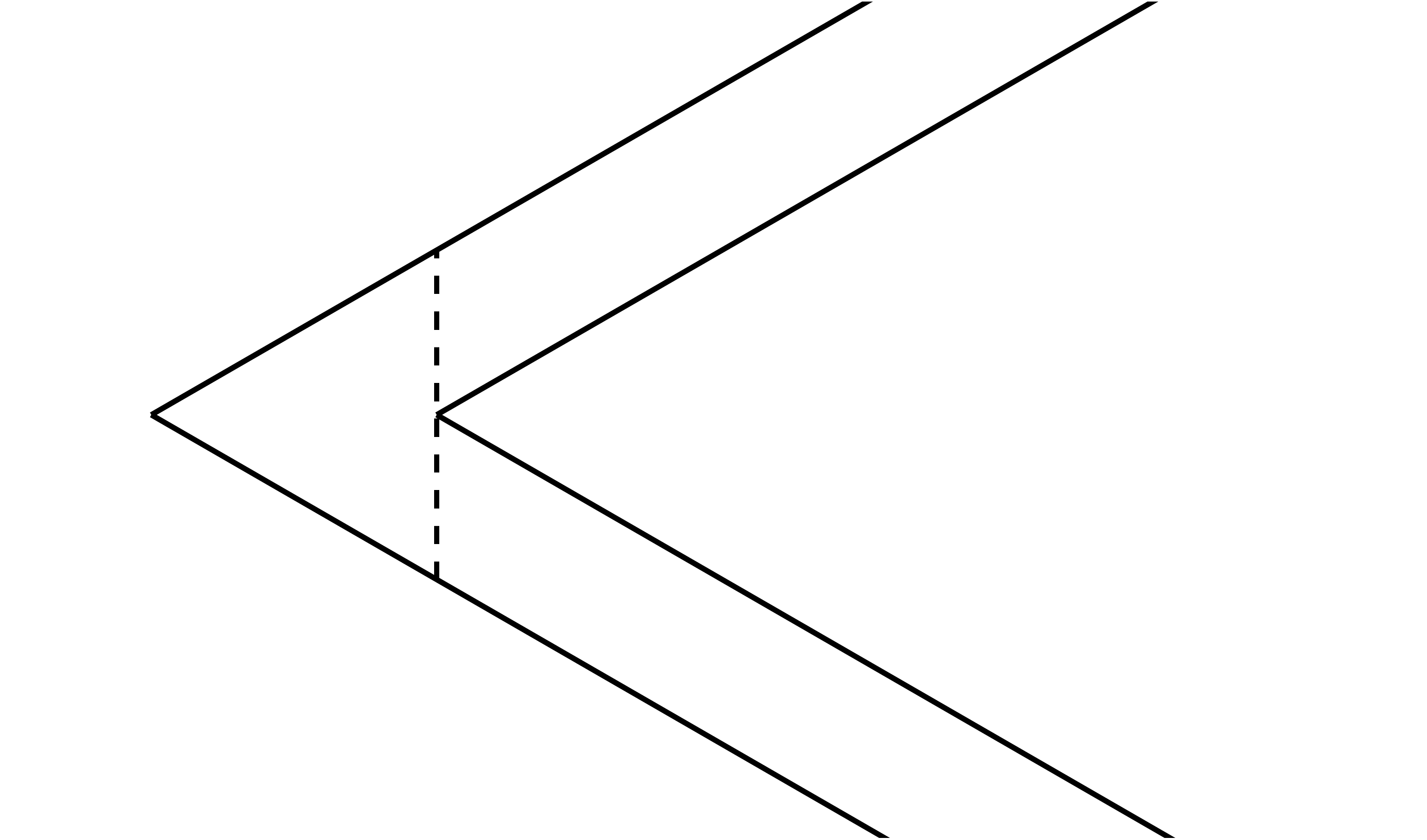}
\caption{The set of Example \ref{exa:strip}. In dashed line, the triangle $T$ such that $\lambda_1(T)\le \mathcal{E}(\Omega)$.}
\label{fig:guida_angolare}
\end{figure}
By \cite[Theorem 1.1]{Si}, we have the following P\'olya--type upper bound for triangles 
\[
\lambda_1(T)\le \left(\frac{\pi}{3}\right)^2\,\left(\frac{L}{A}\right)^2,
\]
where $L$ and $A$ are the perimeter and the area of $T$, respectively. It is easily seen that 
\[
L=\frac{2}{\sin\vartheta\,\cos\vartheta}+\frac{2}{\cos\vartheta}\qquad \mbox{ and }\qquad A=\frac{1}{\sin\vartheta\,\cos\vartheta},
\]
thus we have
\[
\lambda_1(T)\le \left(\frac{\pi}{3}\right)^2\,\left(2+2\,\sin\vartheta\right)^2.
\]
In particular, for every $0<\vartheta\le \pi/6$ we have 
\[
\lambda_1(T)\le \left(\frac{\pi}{3}\right)^2\,\left(2+2\,\sin\vartheta\right)^2\le \pi^2.
\]
This entails that for all these angles we have 
\[
\lambda_1(T)\le \pi^2=\lambda_1(\mathcal{C}_1)=\lambda_1(\mathcal{C}_2)=\mathcal{E}(\Omega).
\]
Thus, for every $0<\vartheta\le \pi/6$, also this set is a particular case of Example \ref{exa:massive}. Thus,
we can apply Theorem \ref{teo:first_eigen} and get existence of a first eigenfunction. We refer to \cite{ABGM} and \cite{DR} for further studies on this set.
\end{example}

\begin{example}[Hersch's pipe]
\label{exa:pipe}
We take the set $\mathbf{H}$ introduced in \eqref{pipe}.
We have already observed in the introduction that
$\mathbf{H}$ satisfies Assumptions \ref{ass:Omega}, by taking $N=2$, $k=2$ and
\[
\Omega_0=\Big\{(x,y)\in\mathbb{R}^2\,:\, x_1^2+x_2^2<1\Big\}\setminus\Big([0,1]\times \{0\}\Big) \qquad (\mbox{\it slit disk}),
\]
\[
\mathcal{C}_1=(0,+\infty)\times(0,1),\qquad \mathcal{C}_2=(0,+\infty)\times (-1,0).
\]
Moreover, we have already remarked that
\[
\lambda_1(\Omega_0)=\pi^2.
\]
By observing that $\mathcal{E}(\mathbf{H})=\lambda_1(\mathcal{C}_1)=\lambda_1(\mathcal{C}_2)=\pi^2$ and that $\Omega_0\cap \mathcal{C}_i\not=\emptyset$ for both $i=1,2$, we see that $\mathbf{H}$ is actually a particular case of Example \ref{exa:massive}. Thus, we can apply Theorem \ref{teo:first_eigen}, here as well.
\end{example}

\section{Higher eigenvalues}
\label{sec:4}

In the next result, we still denote by $\mathcal{E}(\Omega)$ the threshold energy defined in \eqref{energybelow} and we let $\lambda_j(\Omega)$ be as in \eqref{lambdak}.
\begin{theorem}
\label{teo:higher_eigen}
Let $\Omega\subseteq\mathbb{R}^N$ be an open set satistying Assumptions \ref{ass:Omega}. If $\ell\in\mathbb{N}\setminus\{0,1\}$ is such that
\[
	\lambda_\ell(\Omega)<\mathcal{E}(\Omega),
\]
then for every $j=1,\dots,\ell$ we have:
\begin{enumerate}
\item $\lambda_j(\Omega)$ is an eigenvalue, with associated eigefunction $u_j\in \mathcal{S}_2(\Omega)$. Moreover, $\{u_1,\dots,u_j\}$ can be chosen to form an orthonormal set in $L^2(\Omega)$ and we have
\[
\lambda_j(\Omega)=\min_{u\in \mathcal{S}_2(\Omega)}\left\{\int_\Omega |\nabla u|^2\,dx\,:\,\int_\Omega u\,u_m\,dx=0,\ \mbox{for }m=1,\dots,j-1\right\},
\]
with the minimum attained by $u_j$;
\vskip.2cm
\item if $\{\Omega_n\}_{n\in\mathbb{N}}$ is the exhausting sequence of sets defined in \eqref{Omegar}, for every $n\in\mathbb{N}$ there exists $\{u^n_1,\dots,u^n_j\}\subseteq \mathcal{S}_2(\Omega_n)$ eigenfunctions of $\Omega$ associated to $\lambda_1(\Omega_n),\dots,\lambda_j(\Omega_n)$ which forms an orthonormal set in $L^2(\Omega_n)$ and such that (up to a subsequence)
\[
\lim_{n\to\infty} \|u^n_m-u_m\|_{L^2(\Omega)}=0,\qquad \mbox{ for } m=1,\dots,j.
\]
\end{enumerate} 
\end{theorem}
\begin{proof}
	We will prove the result by finite induction on $j\in \{1,\dots,\ell\}$. By observing that (see Lemma \ref{lemma:crescono})
	\[
	\lambda_1(\Omega)\le \lambda_\ell(\Omega)<\mathcal{E}(\Omega),
	\]
	we obtain that the case $j=1$ has been proved in Theorem \ref{teo:first_eigen}.
	\par 
	We now assume that the statement holds for every $j=1,\dots,\ell-1$. We need to show that this entails the validity of the case $j=\ell$, as well. In order to do this, we will adapt the same idea of Theorem \ref{teo:first_eigen}, by taking into account the additional difficulties connected with the orthogonality relations.
	\par	
The inductive assumptions implies that $\lambda_1(\Omega),\dots,\lambda_{\ell-1}(\Omega)$ are eigenvalues of $\Omega$, with an orthonormal set of associated eigenfunctions $\{u_1,\dots,u_{\ell-1}\}\subseteq \mathcal{S}_2(\Omega)$. Moreover, we know that these eigenfunctions can be approximated, strongly in $L^2(\Omega)$, by an orthonormal set $\{u_1^n,\dots,u^n_{\ell-1}\}\subseteq \mathcal{S}_2(\Omega_n)$ of eigenfunctions of $\Omega_n$. We then observe that (see \cite[Chapter VI]{CH})
\[
	\begin{split}
		\lambda_\ell(\Omega_n)&=\inf\left\{ \max_{u\in F\cap \mathcal{S}_2(\Omega_n)} \int_{\Omega_n}|\nabla u|^2\,dx\,:\, F \subseteq W^{1,2}_0(\Omega_n) \mbox{ subspace with} \dim F=\ell \right\}\\
		&=\inf_{u\in  \mathcal{S}_2(\Omega_n)}\left\{\int_{\Omega_n} |\nabla u|^2\,dx\,:\,\int_\Omega u\,u^n_j=0,\,\mbox{for }j=1,\dots,\ell-1\right\}.
	\end{split}
\]
We take $u_\ell^n$ to be a minimizer of the last problem, thus by construction $\{u_1^n,\dots,u^n_{\ell-1},u^n_\ell\}$ is an orthonormal set. Existence of $u_\ell^n$ is a plain consequence of the compactness of the embedding $W^{1,2}_0(\Omega_n)\hookrightarrow L^2(\Omega_n)$. By Lemma \ref{lm:limlamk}, we have that
\[
\lim_{n\to\infty} \int_{\Omega_n} |\nabla u_\ell^n|^2\,dx=\lim_{n\to\infty} \lambda_\ell(\Omega_n)=\lambda_\ell(\Omega),
\]
thus the sequence $\{u^n_\ell\}_{n\ge r_0}\subseteq\mathcal{S}_2(\Omega_n)$ is bounded in $W^{1,2}_0(\Omega)$. This implies that, up to a subsequence, it weakly converges in $W^{1,2}(\Omega)$ to a function $u_\ell\in W^{1,2}_0(\Omega)$. We claim that $u_\ell$ is an eigenfunction associated to $\lambda_\ell(\Omega)$, such that
\begin{equation}
	\label{clavinet3}		
			\lim_{n\to \infty}\|u_\ell-u^n_\ell\|_{L^2(\Omega)}=0,
		\end{equation}
		\begin{equation}
		\label{clavinet}
			\int_\Omega u_\ell\,u_j=0,\mbox{ for all }j=1,\dots,\ell-1\qquad \mbox{ and }\qquad \int_\Omega |u_\ell|^2\,dx=1,
\end{equation}	
and
\begin{equation}
\label{clavinet2}
\int_\Omega |\nabla u_\ell|^2\,dx=\min_{\varphi\in \mathcal{S}_2(\Omega)}\left\{\int_\Omega |\nabla \varphi|^2\,dx\,:\,\int_\Omega \varphi\,u_j=0,\,\mbox{for }m=1,\dots,\ell-1\right\},
\end{equation}	
These facts would be sufficient to conclude the proof.
\par
We start from \eqref{clavinet3}: this strong convergence can be inferred by repeating verbatim the compactness argument in the proof of Theorem \ref{teo:first_eigen}. We need to rely this time on the equation for $u^n_\ell$. 
Observe that we already know that the sequence weakly converges to $u_\ell$, thus once the strong compactness is obtained, the strong limit must be the same. Thus \eqref{clavinet3} is established. 
\par
This in particular implies that the second property in \eqref{clavinet} holds true. As for the orthogonality conditions in \eqref{clavinet}, for $j=1,\dots,\ell-1$ we have
\[
		\begin{split}
			\left|\int_\Omega u_\ell\,u_j\,dx\right|&=\lim_{n\to \infty}\left|\int_\Omega u_\ell^n\,u_j\,dx\right|
			=\lim_{n\to\infty}\left|\int_\Omega u_\ell^n\,(u_j-u_j^n)\,dx\right|\le\lim_{n\to\infty}\|u_j-u_j^n\|_{L^2(\Omega)}=0.
		\end{split}
		\]
		Observe that in the second equality we used that $u^n_\ell$ is orthogonal to $u_j^n$ (see above). 	
\par
In order to show that $u_\ell$ is an eigenfunction associated to $\lambda_\ell(\Omega)$, we first observe that $u_\ell$ is non-trivial, thanks to the normalization condition on the $L^2(\Omega)$ norm. It is then sufficient to pass to the limit in the equation for $u^n_\ell$. Indeed, $\{\Omega_n\}_{n\ge r_0}$ is an exhausting sequence for $\Omega$. Thus for every $\varphi\in C^\infty_0(\Omega)$ we have that this is compactly supported in $\Omega_n$, as well, for $n$ large enough (depending on $\varphi$). Hence, from the weak convergence of $\{u^n_\ell\}_{n\in\mathbb{N}}$ and Lemma \ref{lm:limlamk}, for every fixed $\varphi\in C^\infty_0(\Omega)$ we get
		\[
		0=\lim_{n\to\infty}\left(\int_{\Omega_n}\langle \nabla u^n_\ell,\,\nabla\varphi\rangle\,dx-\lambda_\ell(\Omega_n)\,\int_{\Omega_n}u^n_\ell\,\varphi\,dx\right)=\int_{\Omega}\langle \nabla u_\ell,\,\nabla\varphi\rangle\,dx-\lambda_\ell(\Omega)\,\int_{\Omega}u_\ell\,\varphi\,dx.
		\]
This is valid for every $\varphi\in C^\infty_0(\Omega)$, thus by density we get that $u_\ell$ is an eigenfunction, as claimed.
\par
We are only left with proving \eqref{clavinet2}. From \eqref{clavinet} and the fact that $u_\ell$ is an eigenfunction, we already know that 
\[
\lambda_\ell(\Omega)=\int_\Omega |\nabla u_\ell|^2\,dx\ge \inf_{\varphi\in \mathcal{S}_2(\Omega)}\left\{\int_\Omega |\nabla \varphi|^2\,dx\,:\,\int_\Omega \varphi\,u_j=0,\,\mbox{for }m=1,\dots,\ell-1\right\}.
\]
On the other hand, for every $u\in \mathcal{S}_2(\Omega)$ which is orthogonal to $u_1,\dots,u_{\ell-1}$ in $L^2(\Omega)$, we can consider the
$\ell-$dimensional vector subspace of $W^{1,2}_0(\Omega)$ generated by $\{u_1,\dots,u_{\ell-1},u\}$. Then by definition we have
\[
\begin{split}
\lambda_\ell(\Omega)\le \max_{a\in\mathbb{S}^{\ell-1}} \int_\Omega \left|\sum_{j=1}^{\ell-1} a_j\,\nabla u_j+a_\ell\,\nabla u\right|^2\,dx&=\max_{a\in\mathbb{S}^{\ell-1}} \left(\sum_{j=1}^{\ell-1} a_j^2\,\int_\Omega |\nabla u_j|^2\,dx+a_\ell^2\,\int_\Omega |\nabla u|^2\,dx\right)\\
&=\max_{a\in\mathbb{S}^{\ell-1}} \left(\sum_{j=1}^{\ell-1} a_j^2\,\lambda_j(\Omega)+a_\ell^2\,\int_\Omega |\nabla u|^2\,dx\right).
\end{split}
\]
Here we have used the orthogonality conditions and the fact that 
\[
\int_\Omega \langle \nabla u_j,\nabla u\rangle\,dx=\lambda_j(\Omega)\,\int_\Omega u_j\,u\,dx,\qquad \mbox{ for } j=1,\dots,\ell-1.
\]
We now use the inductive assumption to assure that for $j=1,\dots,\ell-1$
\[
\lambda_j(\Omega)=\inf_{\varphi\in \mathcal{S}_2(\Omega)}\left\{\int_\Omega |\nabla \varphi|^2\,dx\,:\,\int_\Omega \varphi\,u_m=0,\,\mbox{for }m=1,\dots,j-1\right\}\le \int_\Omega |\nabla u|^2\,dx.
\]
Thus, from the estimate above we get
\[
\lambda_\ell(\Omega)\le \max_{a\in\mathbb{S}^{\ell-1}} \left(\sum_{j=1}^{\ell-1} a_j^2\,\lambda_j(\Omega)+a_\ell^2\,\int_\Omega |\nabla u|^2\,dx\right)\le \int_\Omega |\nabla u|^2\,dx.
\]
By recalling that $u$ is an arbitrary trial functions for the minimization problem in the right-hand side of \eqref{clavinet2}, we get that
\[
\lambda_\ell(\Omega)\le \inf_{u\in \mathcal{S}_2(\Omega)}\left\{\int_\Omega |\nabla u|^2\,dx\,:\,\int_\Omega u\,u_j=0,\,\mbox{for }m=1,\dots,\ell-1\right\},
\]
as well. Thus \eqref{clavinet2} holds true: observe that we also obtained that the infimum in \eqref{clavinet2} is attained by $u_\ell$. 
\par
The proof is now over.
\end{proof}

\section{Some estimates for eigenfunctions}
\label{sec:5}

In this section, we prove that the eigenfunctions obtained in Sections \ref{sec:3} and \ref{sec:4} 
satisfy suitable decay estimates at infinity.

\begin{theorem}
\label{teo:eigenproperties}
Let $\Omega\subseteq\mathbb{R}^N$ be an open set satistying Assumptions \ref{ass:Omega}. We suppose that 
\[
\lambda_\ell(\Omega)<\mathcal{E}(\Omega),
\]
for some $\ell\in\mathbb{N}\setminus\{0\}$.
For every $j=1,\dots,\ell$, there exists $0<\beta_j<1$ and two constants $C_{1,j},C_{2,j}>0$ such that for every eigenfunction $u_j$ associated to $\lambda_j(\Omega)$ we have    
	\begin{equation}
	\label{decayL2}
		\|u_j\|_{L^2(\Omega \setminus \Omega_R)}\le C_{1,j}\,\|u_j\|_{L^2(\Omega)}\,\beta_j^{R},\qquad \mbox{ for every } R>0,
	\end{equation}
	and
	\begin{equation}
	\label{decayLinfty}
		|u_j(x)|\le C_{2,j}\,\|u_j\|_{L^2(\Omega)}\,\beta_j^{R},\qquad \mbox{ for every } R>0  \mbox{ and for a.\,e. } |x|>R.
	\end{equation}
Moreover, we also have
	\begin{equation}
	\label{stimalinfa}
		\|u_j\|_{L^\infty(\Omega)}\le C_3\,\Big(\lambda_j(\Omega)\Big)^\frac{N}{4}\,\|u_j\|_{L^2(\Omega)},
	\end{equation}
for a constant $C_3=C_3(N)>0$.	
\end{theorem}
\begin{proof}
We first observe that $|u_j|\in W^{1,2}_0(\Omega)$ is a non-negative function such that 
\begin{equation}
\label{eigen1}
\int_{\Omega} \langle \nabla |u_j|,\nabla \varphi\rangle\,dx\le \lambda\,\int_{\Omega} |u_j|\,\varphi\,dx,\qquad \mbox{ for every } \varphi\in W^{1,2}_0(\Omega) \ \mbox{ such that } \varphi\ge 0,
\end{equation}
thanks to Lemma \ref{lemma:sottosoluzione}.
We divide the proof in three parts, according to the property of $u_j$ we need to prove.
\vskip.2cm\noindent
\underline{\it Decay in $L^2$.} In order to prove \eqref{decayL2}, we first observe that $u_j$ satisfies the following estimate
	\[
		\Big((1-\delta)\,\mathcal{E}(\Omega)-\lambda_j(\Omega)\Big)\,\int_{\mathcal{C}_i\setminus \Omega^i_{r+1}} |u_j|^2\,dx\leq \left(\frac{4}{\delta}+\lambda_j(\Omega)\right)\,\int_{\Omega^i_{r+1}\setminus\Omega^i_r} |u_j|^2\,dx.
	\]
Indeed, it is sufficient to start from \eqref{eigen1} and reproduce verbatim the proof of \eqref{scappotta}, from the proof of Theorem \ref{teo:first_eigen}. We then choose 
\[
\delta=\frac{\mathcal{E}(\Omega)-\lambda_j(\Omega)}{2\,\mathcal{E}(\Omega)}.
\]
With simple manipulations, along the lines we used to get \eqref{predecay}, we now obtain
\[
\int_{\mathbb{R}^N\setminus \Omega_{r+1}} |u_j|^2\,dx=\sum_{i=1}^k\int_{\mathcal{C}_i\setminus \Omega^i_{r+1}} |u_j|^2\,dx\le C_{\Omega,j}\,\int_{\Omega_{r+1}\setminus\Omega_r} |u_j|^2\,dx.
\]
The constant $C_{\Omega,j}>0$ is given by 
\[
C_{\Omega,j}:=\frac{2}{\mathcal{E}(\Omega)-\lambda_j(\Omega)}\,\left(\frac{8\,\mathcal{E}(\Omega)}{\mathcal{E}(\Omega)-\lambda_j(\Omega)}+\mathcal{E}(\Omega)\right).
\]
With same argument as in the proof of Theorem \ref{teo:first_eigen}, the previous estimate permits to infer that
\[
\int_{\mathbb{R}^N\setminus \Omega_{R+1}} |u_j|^2\,dx\le \left(\frac{C_{\Omega,j}}{C_{\Omega,j}+1}\right)^{R-r_0}\,\int_\Omega |u_j|^2\,dx,\qquad \mbox{ for every } R\ge r_0.
\] 
By setting 
\[
\beta_j=\sqrt{\frac{C_{\Omega,j}}{C_{\Omega,j}+1}}\qquad \mbox{ and }\qquad C_{1,j}=\left(\frac{C_{\Omega,j}}{C_{\Omega,j}+1}\right)^{-\frac{r_0}{2}-1},
\]
we get \eqref{decayL2} for $R\ge r_0+1$. On the other hand, for $0<R<r_0+1$ it is sufficient to notice that
\[
\|u_j\|_{L^2(\Omega \setminus \Omega_R)}\le \|u_j\|_{L^2(\Omega)}\le \|u_j\|_{L^2(\Omega)}\,\left(\frac{1}{\beta_j^{r_0+1}}\right)\,\beta_j^R,
\]
with the same choice of $0<\beta_j<1$ as above.
\vskip.2cm\noindent
\underline{\it Global boundedness}. We now prove that $u_j\in L^\infty(\Omega)$. In what follows, for simplicity we will simply write $u$ in place of $|u_j|$.
For every $\gamma\ge 1$ and $M>0$, we insert in \eqref{eigen1} the test function
\[
\varphi=u_M^\gamma:=\left(\min\{u,M\}\right)^\gamma.
\]
Thanks to the Chain Rule in $W^{1,2}_0(\Omega)$, this is a feasible test function. With standard computations, we obtain
\begin{equation}
\label{cub}
\frac{4\,\gamma}{(\gamma+1)^2}\,\int_\Omega \left|\nabla u_M^\frac{\gamma+1}{2}\right|^2\,dx=\lambda_j(\Omega)\,\int_\Omega u\,u_M^\gamma.
\end{equation}
In order to estimate from below the left-hand side, we will use the following two functional inequalities: for $N\ge 3$, the {\it Sobolev inequality} (see \cite{TaG})
\begin{equation*}
T_N\,\left(\int_{\mathbb{R}^N} |\varphi|^{2^*}\,dx\right)^\frac{2}{2^*}\le \int_{\mathbb{R}^N} |\nabla \varphi|^2,\qquad \mbox{ for } T_N=N\,(N-2)\,\pi\,\left(\frac{\Gamma(N/2)}{\Gamma(N)}\right)^\frac{2}{N},
\end{equation*}
while for $N=2$ we will use the {\it Ladyzhenskaya inequality} (see \cite[equation (1.11)]{Le})
\begin{equation}\label{eq:lade}
\pi\,\left(\int_{\mathbb{R}^2} |\varphi|^4\,dx\right)\le \left(\int_{\mathbb{R}^2} |\nabla \varphi|^2\,dx\right)\,\left(\int_{\mathbb{R}^2} |\varphi|^2\,dx\right),
\end{equation}
both holding for every $\varphi\in W^{1,2}(\mathbb{R}^N)$. By sticking for the moment to the case $N\ge 3$, we then obtain
\[
\frac{4\,\gamma}{(\gamma+1)^2}\,T_N\,\left(\int_\Omega \left|u_M^\frac{\gamma+1}{2}\right|^{2^*}\,dx\right)^\frac{2}{2^*}\le \lambda_j(\Omega)\,\int_\Omega u\,u_M^\gamma.
\]
We now introduce the parameter
\[
\vartheta=\frac{\gamma+1}{2}\ge 1,
\] 
and rewrite the previous estimate as
\[
\begin{split}
\left(\int_\Omega u_M^{2^*\vartheta}\,dx\right)^\frac{2}{2^*}&\le\frac{\vartheta^2}{2\,\vartheta-1}\,\frac{\lambda_j(\Omega)}{T_N}\,\int_\Omega u\,u_M^{2\,\vartheta-1}\le \vartheta\,\frac{\lambda_j(\Omega)}{T_N}\,\int_\Omega u\,u_M^{2\,\vartheta-1}.
\end{split}
\]
We then define the recursive sequence of exponents
\begin{equation}
\label{sequenza}
\vartheta_0=1,\qquad \vartheta_{i+1}=\frac{2^*}{2}\,\vartheta_i=\left(\frac{N}{N-2}\right)^{i+1}.
\end{equation}
Moreover, with simple manipulations we also get
\begin{equation}
\label{schema}
\left(\int_\Omega u_M^{2\vartheta_{i+1}}\,dx\right)^\frac{1}{2\vartheta_{i+1}}\le \vartheta_i^\frac{1}{2\vartheta_i}\,\left(\frac{\lambda_j(\Omega)}{T_N}\right)^\frac{1}{2\vartheta_i}\,\left(\int_\Omega u\,u_M^{2\,\vartheta_i-1}\right)^\frac{1}{2\vartheta_i}.
\end{equation}
We claim at first that \eqref{schema} shows that $u\in L^{2\vartheta_i}(\Omega)$, for every $i\in\mathbb{N}$. We prove this fact by induction: for $i=0$, we have $2\vartheta_0=2$ and $u\in L^2(\Omega)$ holds true by assumption. Let us now assume that $u\in L^{2\vartheta_i}(\Omega)$: this assumption, \eqref{schema} and $u_M\le u$ imply that 
\[
\left(\int_\Omega u_M^{2\vartheta_{i+1}}\,dx\right)^\frac{1}{2\vartheta_{i+1}}\le \vartheta_i^\frac{1}{2\vartheta_i}\,\left(\frac{\lambda_j(\Omega)}{T_N}\right)^\frac{1}{2\vartheta_i}\,\left(\int_\Omega u^{2\,\vartheta_i}\right)^\frac{1}{2\vartheta_i}.
\]
The right-hand side is finite and independent of $M>0$: by taking the limit as $M$ goes to $\infty$ and using Fatou's Lemma, we then obtain
\begin{equation}
\label{schema2}
\left(\int_\Omega u^{2\vartheta_{i+1}}\,dx\right)^\frac{1}{2\vartheta_{i+1}}\le \vartheta_i^\frac{1}{2\vartheta_i}\,\left(\frac{\lambda_j(\Omega)}{T_N}\right)^\frac{1}{2\vartheta_i}\,\left(\int_\Omega u^{2\,\vartheta_i}\right)^\frac{1}{2\vartheta_i}.
\end{equation}
Thus $u\in L^{2\vartheta_{i+1}}(\Omega)$, as well.
\par
By observing that $\{\vartheta_i\}_{i\in\mathbb{N}}$ is a increasingly diverging sequence, we then obtain in particular that $u\in L^\gamma(\Omega)$, for every $2\le \gamma<+\infty$. In order to obtain the boundedness of $u$, it is now sufficient to use recursively the scheme of reverse H\"older inequalities \eqref{schema2}: after $n$ iterations, we get
\begin{equation}
\label{schema3}
\left(\int_\Omega u^{2\vartheta_n}\,dx\right)^\frac{1}{2\vartheta_n}\le \left(\prod_{i=0}^{n-1} \vartheta_i^\frac{1}{2\vartheta_i}\right)\,\left(\frac{\lambda_j(\Omega)}{T_N}\right)^{\sum\limits_{i=0}^{n-1}\frac{1}{2\vartheta_i}}\,\left(\int_\Omega u^{2}\,dx\right)^\frac{1}{2}.
\end{equation}
By observing that 
\begin{equation}
\label{limiteexponents}
\lim_{n\to\infty}\sum\limits_{i=0}^{n-1}\frac{1}{2\vartheta_i}=\frac{1}{2}\,\sum_{i=0}^\infty \left(\frac{N-2}{N}\right)^i=\frac{N}{4},
\end{equation}
and that
\[
\lim_{n\to\infty} \left(\prod_{i=0}^{n-1} \vartheta_i^\frac{1}{2\vartheta_i}\right)=:C_N<+\infty,
\]
and passing to the limit as $n\to\infty$ in \eqref{schema3}, we finally get
\[
\|u\|_{L^\infty(\Omega)}\le C_N\,\left(\frac{\lambda_j(\Omega)}{T_N}\right)^\frac{N}{4}\,\|u\|_{L^2(\Omega)}.
\]
This gives the boundedness of $u$, in the case $N\ge 3$.
\par
We briefly describe how to adapt the proof to the case $N=2$: we go back to \eqref{cub} and multiply both sides by the $L^2$ norm of $u_M^{(\gamma+1)/2}$. This gives
\[
\frac{4\,\gamma}{(\gamma+1)^2}\,\left(\int_\Omega \left|\nabla u_M^\frac{\gamma+1}{2}\right|^2\,dx\right)\,\left(\int_\Omega u_M^{\gamma+1}\,dx\right)=\lambda_j(\Omega)\,\left(\int_\Omega u\,u_M^\gamma\,dx\right)\,\left(\int_\Omega u_M^{\gamma+1}\,dx\right).
\]
We use the Ladyzhenskaya inequality on the left-hand side and use again the parameter $\vartheta=(\gamma+1)/2$. We now get
\[
\begin{split}
\int_\Omega u_M^{4\vartheta}\,dx\le \vartheta\,\frac{\lambda_j(\Omega)}{\pi}\,\left(\int_\Omega u\,u_M^{2\vartheta-1}\,dx\right)\,\left(\int_\Omega u_M^{2\vartheta}\,dx\right).
\end{split}
\]
This time we define the recursive sequence of exponents
\[
\vartheta_0=1,\qquad \vartheta_{i+1}=2\,\vartheta_i=2^{i+1},
\]
thus, in place of \eqref{schema}, we now get
\[
\left(\int_\Omega u_M^{2\vartheta_{i+1}}\,dx\right)^\frac{1}{2\vartheta_{i+1}}\le \vartheta_i^\frac{1}{4\vartheta_i}\,\left(\frac{\lambda_j(\Omega)}{\pi}\right)^\frac{1}{4\vartheta_i}\,\left(\int_\Omega u\,u_M^{2\,\vartheta_i-1}\right)^\frac{1}{4\vartheta_i}\,\left(\int_\Omega u_M^{2\vartheta_i}\,dx\right)^\frac{1}{4\vartheta_i}.
\]
We can now proceed as above, to prove at first that $u\in L^\gamma(\Omega)$ for every $2\le \gamma<+\infty$ and then obtain the estimate
\[
\|u\|_{L^\infty(\Omega)}\le C\,\sqrt{\frac{\lambda_j(\Omega)}{\pi}}\,\|u\|_{L^2(\Omega)},
\]
where $C>0$ is a universal constant.
\vskip.2cm\noindent
\underline{\it Decay in $L^\infty$}. In order to upgrade the decay estimate to the $L^\infty$ norm, we will use again a suitable Moser's iteration, this time ``localized at infinity''. We keep on using the notation $u$ in place of $|u_j|$ and we recall that $r_0$ is as in \eqref{eq:r_0}. 
We fix $i=1,\dots,k$ and take $r_0+1\le r<R$, then in the equation \eqref{eigen1} we use the test function
\[
\varphi=u^\gamma\,\left(\eta^{(i)}_{r,R}\right)^2,
\]
where $\eta^{(i)}_{r,R}$ is the ``cut-off at infinity'' defined in Definition \ref{def:cutoff} and $\gamma\ge 1$. Observe that this is a feasible test function, thanks to the fact that $u\in L^\infty(\Omega)$, from the previous step of the proof. Thus, in particular $u^\gamma\in W^{1,2}_0(\Omega)$. From the equation, we get
\begin{equation}
\label{prospero}
\begin{split}
\frac{4\,\gamma}{(\gamma+1)^2}\,\int_{\mathbb{R}^N} \left|\nabla u^\frac{\gamma+1}{2}\right|^2\,\left(\eta^{(i)}_{r,R}\right)^2\,dx&+2\,\int_{\mathbb{R}^N} \left\langle \nabla u,\nabla \eta^{(i)}_{r,R}\right\rangle\,u^\gamma\,\eta^{(i)}_{r,R}\,dx\\
&\le \lambda_j(\Omega)\,\int_{\mathbb{R}^N} u^{\gamma+1}\,\left(\eta^{(i)}_{r,R}\right)^2\,dx.
\end{split}
\end{equation}
By using the Cauchy-Schwarz and Young inequalities, we have for $\delta>0$
\[
\begin{split}
2\,\int_{\mathbb{R}^N} \left\langle \nabla u,\nabla \eta^{(i)}_{r,R}\right\rangle\,u^\gamma\,\eta^{(i)}_{r,R}\,dx&\ge -\delta\, \int_{\mathbb{R}^N} |\nabla u|^2\,u^{\gamma-1}\,\left(\eta^{(i)}_{r,R}\right)^2\,dx-\frac{1}{\delta}\,\int_{\mathbb{R}^N} \left|\nabla \eta^{(i)}_{r,R}\right|^2\,u^{\gamma+1}\,dx\\
&=-\delta\,\frac{4}{(\gamma+1)^2}\,\int_{\mathbb{R}^N}\left|\nabla u^\frac{\gamma+1}{2}\right|^2\,\left(\eta^{(i)}_{r,R}\right)^2\,dx\\
&-\frac{1}{\delta}\,\int_{\mathbb{R}^N} \left|\nabla \eta^{(i)}_{r,R}\right|^2\,u^{\gamma+1}\,dx.
\end{split}
\]
We choose $\delta=\gamma/2$ and insert this estimate in \eqref{prospero}, this gives\footnote{Observe that both integrals on the right-hand side are finite for every $\gamma\ge 1$, thanks to the fact that $u\in L^2(\Omega)\cap L^\infty(\Omega)$.}
\[
\begin{split}
\frac{2\,\gamma}{(\gamma+1)^2}\,\int_{\mathbb{R}^N} \left|\nabla u^\frac{\gamma+1}{2}\right|^2\,\left(\eta^{(i)}_{r,R}\right)^2\,dx&\le \frac{2}{\gamma}\,\int_{\mathbb{R}^N} \left|\nabla \eta^{(i)}_{r,R}\right|^2\,u^{\gamma+1}\,dx\\
&+ \lambda_j(\Omega)\,\int_{\mathbb{R}^N} u^{\gamma+1}\,\left(\eta^{(i)}_{r,R}\right)^2\,dx.
\end{split}
\]
We sum on both sides the quantity
\[
\frac{2\,\gamma}{(\gamma+1)^2}\,\int_{\mathbb{R}^N} \left|\nabla \eta^{(i)}_{r,R}\right|^2\,u^{\gamma+1}\,dx,
\]
and use that 
\[
\int_{\mathbb{R}^N} \left|\nabla u^\frac{\gamma+1}{2}\right|^2\,\left(\eta^{(i)}_{r,R}\right)^2\,dx+\int_{\mathbb{R}^N} \left|\nabla \eta^{(i)}_{r,R}\right|^2\,u^{\gamma+1}\,dx\ge \frac{1}{2}\,\int_{\mathbb{R}^N} \left|\nabla \left(u^\frac{\gamma+1}{2}\,\eta^{(i)}_{r,R}\right)\right|^2\,dx.
\]
We thus obtain
\[
\begin{split}
\frac{\gamma}{(\gamma+1)^2}\,\,\int_{\mathbb{R}^N} \left|\nabla \left(u^\frac{\gamma+1}{2}\,\eta^{(i)}_{r,R}\right)\right|^2\,dx&\le \left[\frac{2}{\gamma}+\frac{2\,\gamma}{(\gamma+1)^2}\,\right]\,\int_{\mathbb{R}^N} \left|\nabla \eta^{(i)}_{r,R}\right|^2\,u^{\gamma+1}\,dx\\
&+ \lambda_j(\Omega)\,\int_{\mathbb{R}^N} u^{\gamma+1}\,\left(\eta^{(i)}_{r,R}\right)^2\,dx. 
\end{split}
\]
We now polish a little bit this estimate. We muplitly both sides by $(\gamma+1)^2/\gamma$ and use that
\[
\frac{(\gamma+1)^2}{\gamma^2}\le 4\qquad \mbox{ and }\qquad \frac{(\gamma+1)^2}{\gamma}\le 2\,(\gamma+1),\qquad \mbox{ for } \gamma\ge 1,
\]
so to get
\[
\begin{split}
\int_{\mathbb{R}^N} \left|\nabla \left(u^\frac{\gamma+1}{2}\,\eta^{(i)}_{r,R}\right)\right|^2\,dx&\le 10\,\int_{\mathbb{R}^N} \left|\nabla \eta^{(i)}_{r,R}\right|^2\,u^{\gamma+1}\,dx+2\,(\gamma+1)\, \lambda_j(\Omega)\,\int_{\mathbb{R}^N} u^{\gamma+1}\,\left(\eta^{(i)}_{r,R}\right)^2\,dx. 
\end{split}
\]
Finally, from this we can obtain
\begin{equation}
\label{sporca}
\begin{split}
\int_{\mathbb{R}^N} \left|\nabla \left(u^\frac{\gamma+1}{2}\,\eta^{(i)}_{r,R}\right)\right|^2\,dx&\le 2\,(\gamma+1)\,\Big(3+\lambda_j(\Omega)\Big)\,\int_{\mathbb{R}^N} \left[\left|\nabla \eta^{(i)}_{r,R}\right|^2+\left(\eta^{(i)}_{r,R}\right)^2\right]\,u^{\gamma+1}\,dx.
\end{split}
\end{equation}
We restrict for simplicity to the case $N\ge 3$. By using Sobolev's inequality in the left-hand side of \eqref{sporca}, we obtain
\begin{equation}
\label{sporca2}
\begin{split}
\left(\int_{\mathbb{R}^N}\left(u^\frac{\gamma+1}{2}\,\eta^{(i)}_{r,R}\right)^{2^*}\,dx\right)^\frac{2}{2^*}&\le \frac{2}{T_N}\,(\gamma+1)\,\Big(3+\lambda_j(\Omega)\Big)\,\int_{\mathbb{R}^N} \left[\left|\nabla \eta^{(i)}_{r,R}\right|^2+\left(\eta^{(i)}_{r,R}\right)^2\right]\,u^{\gamma+1}\,dx.
\end{split}
\end{equation}
We now use the properties of the cut-off functions, encoded by Definition \ref{def:cutoff}. From \eqref{sporca2}, we get for $i=1,\dots,k$ and $r_0\le r<R$
\begin{equation}
\label{sporca3}
\begin{split}
\left(\int_{\mathcal{C}_i\setminus Q_R^i}\left(u^\frac{\gamma+1}{2}\right)^{2^*}\,dx\right)^\frac{2}{2^*}&\le \frac{2}{T_N}\,(\gamma+1)\,\Big(3+\lambda_j(\Omega)\Big)\,\left[\frac{4}{(R-r)^2}+1\right]\int_{\mathcal{C}_i\setminus Q_r^i} u^{\gamma+1}\,dx.
\end{split}
\end{equation}
As in the previous step, we introduce the parameter $\vartheta=(\gamma+1)/2$ and rewrite \eqref{sporca3} as follows
\[
\begin{split}
\left(\int_{\mathcal{C}_i\setminus Q_R^i}u^{2^*\vartheta}\,dx\right)^\frac{1}{2^*\vartheta}&\le \left(\frac{4}{T_N}\,\Big(3+\lambda_j(\Omega)\Big)\right)^\frac{1}{2\vartheta}\vartheta^\frac{1}{2\vartheta}\,\left[\frac{4}{(R-r)^2}+1\right]^\frac{1}{2\vartheta}\,\left(\int_{\mathcal{C}_i\setminus Q_r^i} u^{2\vartheta}\,dx\right)^\frac{1}{2\vartheta}.
\end{split}
\]
We take $R_0\ge r_0$ and use the previous estimate with the choices
\[
r=r_j:=(R_0+2)-\frac{1}{2^j},\qquad R=r_{j+1},\qquad \mbox{ for } j\in\mathbb{N},
\]
and $\vartheta=\vartheta_j$, where the latter is again the sequence defined in \eqref{sequenza}. We thus get
\[
\begin{split}
\left(\int_{\mathcal{C}_i\setminus Q_{r_{j+1}}^i}u^{2 \vartheta_{j+1}}\,dx\right)^\frac{1}{2\vartheta_{j+1}}&\le \left(\frac{4}{T_N}\,\Big(3+\lambda_j(\Omega)\Big)\right)^\frac{1}{2\vartheta_j}\vartheta_j^\frac{1}{2\vartheta_j}\,\left(16\cdot 4^j+1\right)^\frac{1}{2\vartheta_j}\\
&\times\,\left(\int_{\mathcal{C}_i\setminus Q_{r_j}^i} u^{2\vartheta_j}\,dx\right)^\frac{1}{2\vartheta_j}.
\end{split}
\]
We iterate this estimate, starting from $j=0$: after $n$ steps, we get
\begin{equation}
\label{ready??}
\begin{split}
\left(\int_{\mathcal{C}_i\setminus Q_{r_n}^i}u^{2 \vartheta_{n}}\,dx\right)^\frac{1}{2\vartheta_{n}}&\le \left(\frac{4}{T_N}\,\Big(3+\lambda_j(\Omega)\Big)\right)^{\frac{1}{2}\sum\limits_{j=0}^{n-1}\frac{1}{\vartheta_j}}\,\prod_{j=0}^{n-1}\left\{\frac{1}{\vartheta_j}\,(16\cdot 4^j+1)\right\}^\frac{1}{2\vartheta_j}\\
&\times\,\left(\int_{\mathcal{C}_i\setminus Q_{R_0+1}^i} u^{2}\,dx\right)^\frac{1}{2}.
\end{split}
\end{equation}
We now pass to the limit as $n$ goes to $\infty$ in \eqref{ready??}. By recalling \eqref{limiteexponents} and observing that
\[
\lim_{n\to\infty}\prod_{j=0}^{n-1}\left\{\frac{1}{\vartheta_j}\,(16\cdot 4^j+1)\right\}^\frac{1}{2\vartheta_j}=:C_N<+\infty,\qquad \lim_{n\to\infty} r_n=R_0+2,
\]
so to get
\[
\begin{split}
\|u\|_{L^\infty(\mathcal{C}_i\setminus Q_{R_0+2}^i)}&\le C\,\left(\frac{3+\lambda_j(\Omega)}{T_N}\right)^\frac{N}{4}\,\left(\int_{\mathcal{C}_i\setminus Q_{R_0+1}^i} u^2\,dx\right)^\frac{1}{2}\\
&\le C\,\left(\frac{3+\lambda_j(\Omega)}{T_N}\right)^\frac{N}{4}\,\left(\int_{\Omega\setminus \Omega_{R_0+1}} u^2\,dx\right)^\frac{1}{2}.
\end{split}
\]
Since this holds for every $i=1,\dots,k$, we finally get for every $R_0\ge r_0$
\[
\|u\|_{L^\infty(\Omega\setminus\Omega_{R_0+2})}\le C\,\left(\frac{3+\lambda_j(\Omega)}{T_N}\right)^\frac{N}{4}\,\left(\int_{\Omega\setminus \Omega_{R_0+1}} u^2\,dx\right)^\frac{1}{2}.
\]
This is now enough to conclude the proof for $N\geq 3$. For $N=2$ we may conclude, analogously to the previous step, by using Ladyzhenskaya's inequality \eqref{eq:lade} in place of the Sobolev inequality.
\end{proof}
\begin{remark}[Quality of the constants]\label{rmk:constants}
By inspecting the proof of the previous result, we easily see that both the base $\beta_j$ and the constant $C_{1,j}$ depend on $\lambda_j(\Omega)$ and on the crucial gap
\[
\mathcal{E}(\Omega)-\lambda_j(\Omega)>0.
\]
In particular, we have
\[
\beta_j\nearrow 1 \quad \mbox{ and }\quad C_{1,j}\nearrow 1,\qquad \mbox{ as } \mathcal{E}(\Omega)-\lambda_j(\Omega)\searrow 0.
\]
The estimate \eqref{stimalinfa} is classical, here the constant $C_3$ only depends on the dimension $N$, through the sharp constant in the Sobolev inequality. Finally, the constant $C_3$ depends on $N$, the first eigenvalue $\lambda_1(\Omega)$ and the constant $C_1$.
\end{remark}

\section{Defect of compactness along the tubes}
\label{sec:6}

In this section, we analyze what happens for energy levels above the threshold $\mathcal{E}(\Omega)$. We will see that compactness is completely lost, in a suitable sense. We will need the definition of {\it constrained Palais-Smale sequence}, given in Definition \ref{def:PS}.
Then the main result of this section is the following
\begin{proposition}
\label{prop:PS}
Let $\Omega\subseteq\mathbb{R}^N$ be an open set satisfying Assumptions \ref{ass:Omega}. Then every $\lambda$ such that
\[
\lambda\ge \mathcal{E}(\Omega)=\min\Big\{\lambda_1(E_1),\dots,\lambda_1(E_k)\Big\},
\] 
admits a constrained Palais-Smale sequence which is weakly converging to $0$.
\end{proposition}
\begin{proof}
We give an explicit construction of such a sequence. Without loss of generality, we can assume that 
\[
\mathcal{E}(\Omega)=\lambda_1(E_1).
\]
Let us indicate by $\psi_1\in W^{1,2}_0(E_1)$ a first positive eigenfunction of $E_1$. For simplicity, we take it normalized, i.e. with unit $L^2$ norm. Up to a rigid movement, we can suppose that 
\[
E_1\subseteq \mathbb{R}^{N-1}\times\{0\}\qquad \mbox{ and }\qquad \mathcal{C}_1=E_1\times(0,+\infty).
\]
By using the notation $x=(x',x_N)\in\mathbb{R}^{N-1}\times\mathbb{R}$, for $n\in\mathbb{N}\setminus\{0\}$ we set $\lambda_n=\lambda+1/n$ and we take the function
\[
\varphi_n=\psi_1(x')\,\zeta_n(x_N),
\]
where 
\[
\zeta_n(t)=\left\{\begin{array}{rl}
\sin\Big(\sqrt{\lambda_n-\lambda_1(E_1)}\,(t-r_0)\Big),& \mbox{ if } r_0<t<r_{0}+R_n:=r_0+\dfrac{2\,n\,\pi}{\sqrt{\lambda_n-\lambda_1(E_1)}},\\
&\\
0,& \mbox{ otherwise},
\end{array}
\right.
\]
and $r_0=r_0(\Omega)$ is the same as in Subsection \ref{sec:3.1}.
\begin{figure}
\includegraphics[scale=.25]{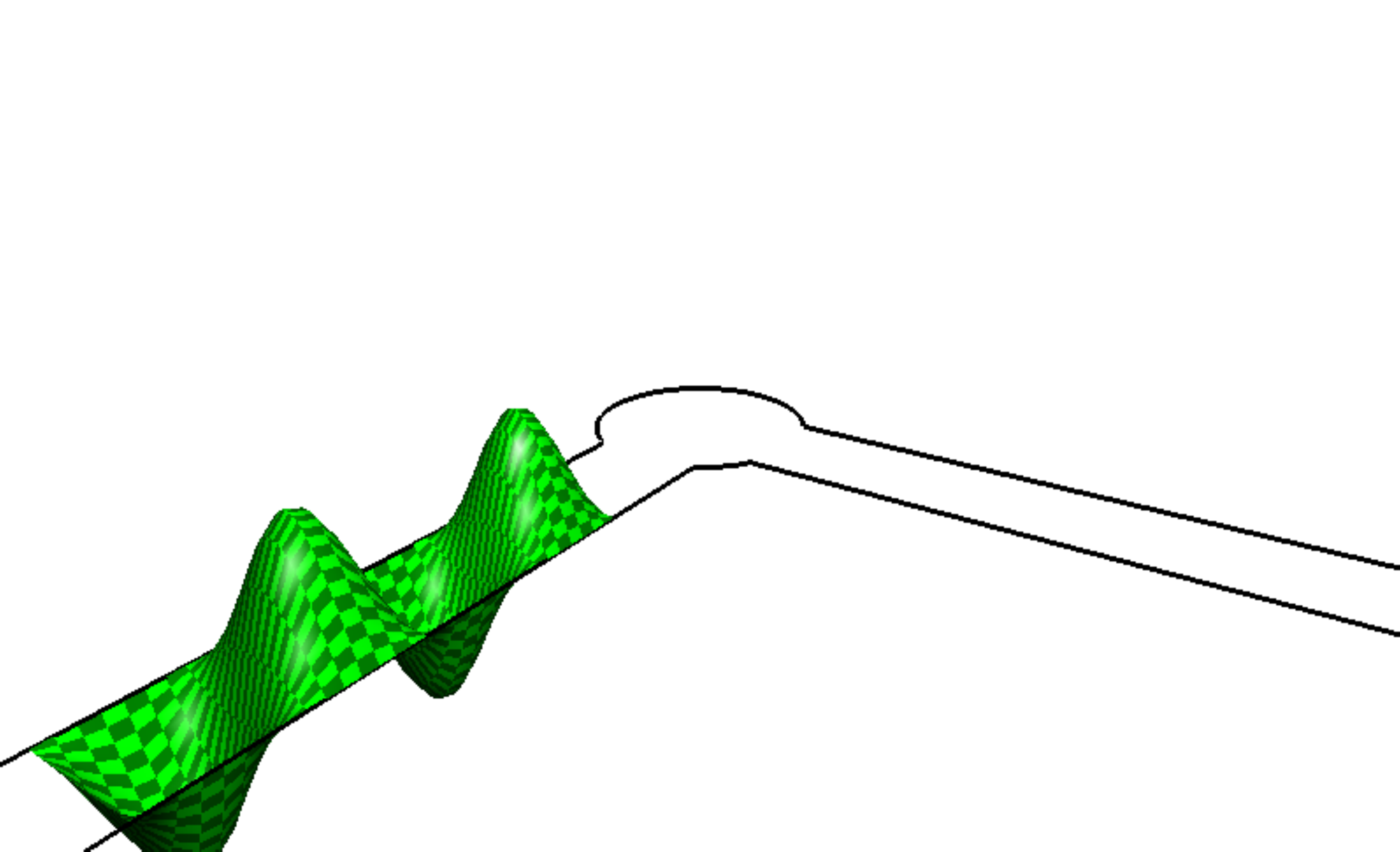}
\caption{The construction of the sequence $\varphi_n$ in the proof of Proposition \ref{prop:PS}.}
\end{figure}
The function $\varphi_n$ is then extended by $0$ to the remainder of $\Omega$, i.e. it is supported in the cylindrical set $\mathcal{C}_1$.
By construction, we have that this weakly solves 
\[
-\Delta \varphi_n=\lambda_n\,\varphi_n,\qquad \mbox{ in } E_1\times \left(r_0,r_0+R_n\right).
\]
Its $L^2$ norm is given by 
\[
\begin{split}
\int_\Omega |\varphi_n|^2\,dx&=\left(\int_{E_1} |\psi_1|^2\,dx'\right)\left(\int_{r_0}^{r_0+R_n} |\zeta_n|^2\,dx_N\right)=\left(\int_0^{R_n} \left(\sin\Big(\sqrt{\lambda_n-\lambda_1(E_1)}\,\tau\Big)\right)^2d\tau\right)=\frac{R_n}{2}.
\end{split}
\]
We then set 
\[
U_n(x)=\frac{\varphi_n(x)}{\|\varphi_n\|_{L^2(\Omega)}}=\sqrt{\frac{2}{R_n}}\,\psi_1(x')\,\zeta_n(x_N).
\]
By construction, it is easily seen that it converges to zero weakly in $L^2(\Omega)$, as well as in $W^{1,2}_0(\Omega)$.
\par
We need to verify that this is a Palais-Smale sequence at the level $\lambda$.
Its Dirichlet integral is given by
\[
\begin{split}
\int_\Omega |\nabla U_n|^2\,dx&=\frac{2}{R_n}\,\left(\int_{E_1} |\nabla' \psi_1|^2\,dx'\right)\,\left(\int_{r_0}^{r_0+R_n} |\zeta_n|^2\,dx_N\right)\\
&+\frac{2}{R_n}\,\left(\int_{E_1} |\psi_1|^2\,dx'\right)\,\left(\int_{r_0}^{r_0+R_n} |\zeta'_n|^2\,dx_N\right)\\
&=\lambda_1(E_1)+\frac{2}{R_n}\,(\lambda_n-\lambda_1(E_1))\int_0^{R_n} \left(\cos\Big(\sqrt{\lambda_n-\lambda_1(E_1)}\,\tau\Big)\right)^2\,d\tau=\lambda_n=\lambda+\frac{1}{n},
\end{split}
\]
where we used that 
\[
\begin{split}
\int_0^{R_n} &\left(\cos\Big(\sqrt{\lambda_n-\lambda_1(E_1)}\,\tau\Big)\right)^2\,d\tau\\
&=\int_0^{R_n} \left(\sin\Big(\sqrt{\lambda_n-\lambda_1(E_1)}\,\tau\Big)\right)^2\,d\tau=\frac{R_n}{2}.
\end{split}
\]
Thus, we obtained that $\{U_n\}_{n\in\mathbb{N}\setminus\{0\}}$ satisfies properties (1) and (2) in Definition \ref{def:PS}. In order to verify point (3), we take $\varphi\in C^\infty_0(\Omega)$. We first observe that by standard Elliptic regularity, we have $\psi_1\in C^\infty(E_1)$ (see for example \cite[Corollary 8.11]{GT}). Thus, in particular, we get that
\[
U_n\in C^2\Big(\overline{\mathcal{O}}\times\left[r_0,r_0+R_n\right]\Big),
\]
for every $\mathcal{O}$ compactly contained in $E_1$. We have enough regularity to justify the following identities: by using the equation satisfied by $U_n$ and the Divergence Theorem, we get
\[
\begin{split}
\int_\Omega \langle \nabla U_n,\nabla \varphi\rangle\,dx&=\int_{E_1}\int_{r_0}^{r_0+R_n} \langle \nabla U_n,\nabla \varphi\rangle\,dx'\,dx_N\\
&=\int_{E_1}\int_{r_0}^{r_0+R_n} \mathrm{div} (\nabla U_n\, \varphi)\,dx'\,dx_N
-\int_{E_1}\int_{r_0}^{r_0+R_n} \Delta U_n\, \varphi\,dx'\,dx_N\\
&=\int_{E_1} \frac{\partial U_n}{\partial x_N}\left(x',r_0+R_n\right)\,\varphi\left(x',r_0+R_n\right)\,dx'\\
&-\int_{E_1} \frac{\partial U_n}{\partial x_N}(x',r_0)\,\varphi(x',r_0)\,dx'+\lambda_n\,\int_\Omega U_n\,\varphi\,dx.
\end{split}
\]
With simple manipulations and by recalling that $\lambda_n=\lambda+1/n$, we get
\[
\begin{split}
\left|\int_\Omega \langle \nabla U_n,\nabla \varphi\rangle\,dx-\lambda\,\int_\Omega U_n\,\varphi\,dx\right|&\le \frac{1}{n}\,\|U_n\|_{L^2(\Omega)}\,\|\varphi\|_{L^2(\Omega)}\\
&+ \left\|\frac{\partial U_n}{\partial x_N}\right\|_{L^2(E_1\times\{r_0\})}\,\|\varphi\|_{L^2(E_1\times \{r_0\})}\\
&+\left\|\frac{\partial U_n}{\partial x_N}\right\|_{L^2\left(E_1\times\{r_0+R_n\}\right)}\,\|\varphi\|_{L^2\left(E_1\times \{r_0+R_n\}\right)}.
\end{split}
\]
By using the properties\footnote{Observe that by construction we have
\[
\frac{\partial U_n}{\partial x_N}(x',r_0)=\sqrt{\frac{2}{R_n}}\,\psi_1(x')\,\zeta'_n(r_0)=\sqrt{\lambda_n-\lambda_1(E_1)}\,\sqrt{\frac{2}{R_n}}\,\psi_1(x'),
\]
and
\[
\frac{\partial U_n}{\partial x_N}(x',r_0+R_n)=\sqrt{\frac{2}{R_n}}\,\psi_1(x')\,\zeta'_n(r_0+R_n)=\sqrt{\lambda_n-\lambda_1(E_1)}\,\sqrt{\frac{2}{R_n}}\,\psi_1(x').
\]} of $U_n$ and the trace inequality for the trial function $\varphi$, we then obtain
\[
\left|\int_\Omega \langle \nabla U_n,\nabla \varphi\rangle\,dx-\lambda\,\int_\Omega U_n\,\varphi\,dx\right|\le \frac{C}{\sqrt{n}}\,\|\varphi\|_{W^{1,2}(\Omega)},
\]
for some $C$ independent of both $n$ and $\varphi$. This finally shows that $\{U_n\}_{n\in\mathbb{N}\setminus\{0\}}$ is a P.-S. sequence at the level $\lambda$. 
\end{proof}

\begin{remark}
\label{remark:essential}
As recalled in the Introduction, a constrained Palais-Smale sequence weakly converging to $0$ can be seen as a variational reformulation of the concept of {\it singular Weyl sequence}, appearing in the Spectral Theory of self-adjoint operators, see for example \cite[Chapter 9, Section 2]{BS} and \cite[Chapter 6, Section 4]{Te}. With this in mind, according to \cite[Theorem 9.2.2]{BS}, the previous result shows that every $\lambda\ge \mathcal{E}(\Omega)$ belongs to the {\it essential spectrum} of the Dirichlet-Laplacian on $\Omega$.
\end{remark}

\begin{remark}
Under the assumptions of the previous result, in general it is not true that {\it every} constrained Palais-Smale sequence at a level $\lambda\ge \mathcal{E}(\Omega)$ weakly converges to $0$. For example, by taking
\[
\Omega=\Big((-1,1)\times(-1,1)\Big) \cup  \Big((2,+\infty)\times (-1,1)\Big),
\]
we see that this set admits an eigenfunction $u\in W^{1,2}_0(\Omega)$, given by the first eigenfunction of $(-1,1)\times(-1,1)$, i.e. associated to the eigenvalue
\[
\lambda=\lambda_1((-1,1)\times(-1,1))=\frac{\pi^2}{2}>\frac{\pi^2}{4}=\lambda_1((2,+\infty)\times (-1,1))=\mathcal{E}(\Omega).
\]
Thus the constant sequence $u_1^n=u$ is a Palais-Smale sequence at a level $\lambda$ larger than $\mathcal{E}(\Omega)$, but of course it is not weakly converging to $0$.
\end{remark}

\section{Singular perturbation of Hersch's pipe}
\label{sec:7}

We indicate by $\mathbf{H}$ the set of Example \ref{exa:pipe}.
We define $\mathrm{I}:=(-1,1)$ and, for $n\in\mathbb{N}\setminus\{0\}$, $0<\varepsilon<1/(2\,n)$ and $i=0,\dots, n$, we denote
\begin{equation*}
	\Sigma^\varepsilon_i:=2+\frac{i}{n}+\varepsilon\, \mathrm{I}=\left(2+\frac{i}{n}-\varepsilon,2+\frac{i}{n}+\varepsilon\right)\subseteq\mathbb{R}\qquad \text{and} \qquad T^\varepsilon_i:=\Sigma_i^\varepsilon\times[1,2).
\end{equation*}
We also denote
\begin{equation*}
	\Gamma^\varepsilon_n:=\bigcup_{i=1}^n(\Sigma^\varepsilon_i\times\{1\})\qquad\text{ and }\qquad	\mathbf{H}^\varepsilon_n:=\mathbf{H}\cup \bigcup_{i=1}^nT^\varepsilon_i.
\end{equation*}
We can observe that the pair $(\mathbf{H},\Gamma^\varepsilon_n)$ satisfies Assumptions \ref{ass:domain} in Appendix \ref{sec:A} with $\Omega=\mathbf{H}$ and $\Sigma=\Gamma^\varepsilon_n$. Accordingly, with the notations of Appendix \ref{sec:A} we have $\Omega_\Sigma=\mathbf{H}^\varepsilon_n$.
\begin{figure}
\includegraphics[scale=.2]{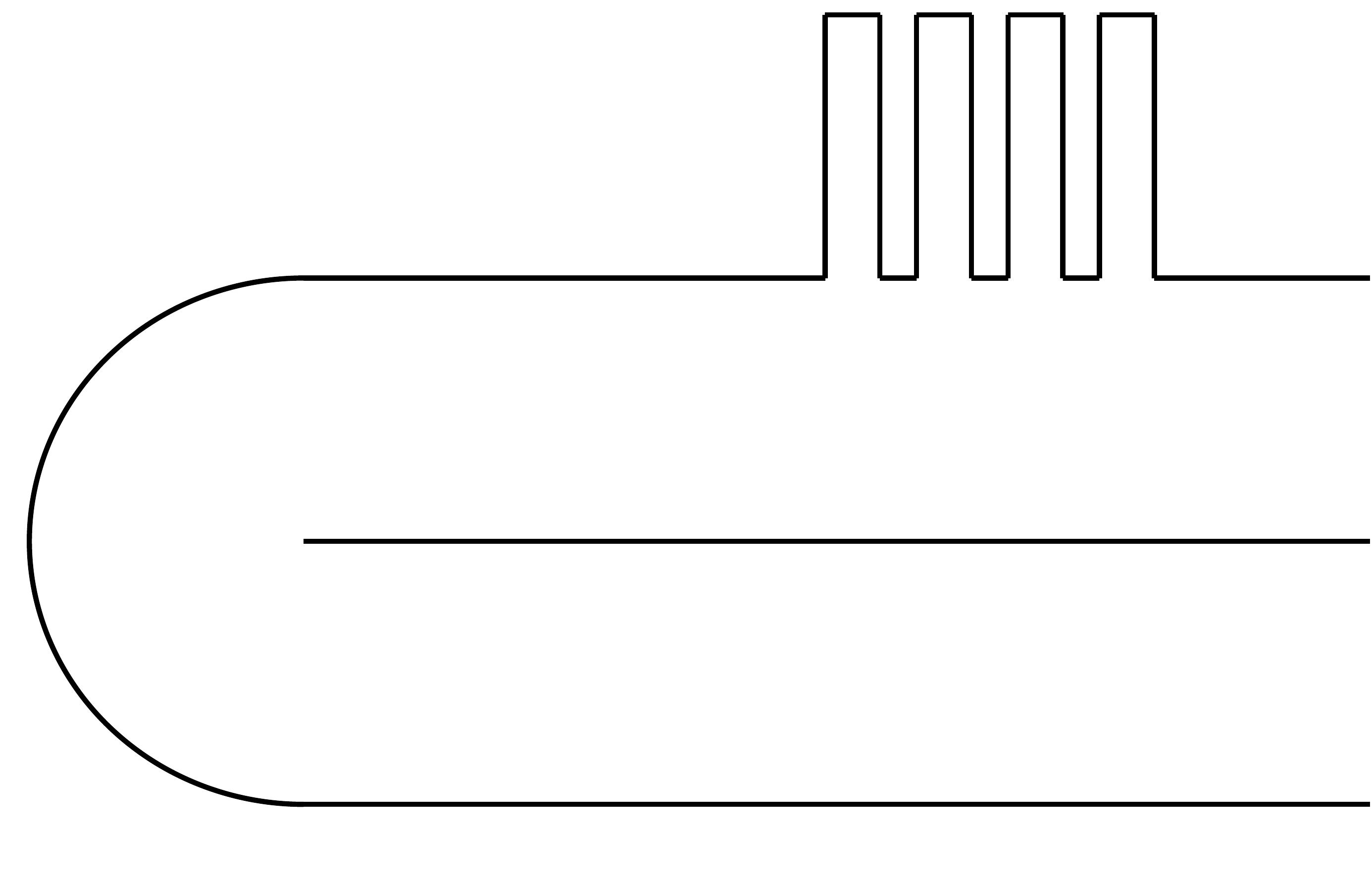}
\caption{The construction of $\mathbf{H}_n^\varepsilon$, with $n=3$: we add to $\mathbf{H}$ an array of $n+1=4$ equally spaced vertical tubes. Each tube has height $1$ and section $2\,\varepsilon$, with $\varepsilon\ll 1$.}
\end{figure}
\vskip.2cm\noindent
The following simple regularity result is not optimal, but it will be largely sufficient for our needs. It asserts that the first eigenfunctions of $\mathbf{H}$ is $C^2$ up to the boundary, in the flat upper part. 
\begin{lemma}
\label{lm:regolaritaminchiona}
Let $u_1\in \mathcal{S}_2(\mathbf{H})$ be the first positive eigenfunction of $\mathbf{H}$. Then for every $0<\tau<L$ we have 
\[
u_1\in C^2\Big(\left[\tau,L\right]\times[\tau,1]\Big).
\]
In particular, the function $\partial u_1/\partial x_2$
has a $C^1$ trace on $(0,+\infty)\times \{1\}$. Moreover, for every $U\in W^{1,2}_0(\mathbf{H}_n^\varepsilon)$ the following identity holds
\[
\int_{\mathbf{H}} \langle \nabla u_1,\nabla U\rangle\,dx=\lambda_1(\mathbf{H})\,\int_{\mathbf{H}} u_1\,U\,dx+\int_{\Gamma^\varepsilon_n} \frac{\partial u_1}{\partial x_2}\,U\,d\mathcal{H}^1.
\]
\end{lemma}
\begin{proof}
For $\tau>0$, we set $A_\tau=(\tau/2,+\infty)\times(0,1)$ and extend $u_1$ to $A'_\tau=(\tau/2,+\infty)\times(0,2)$ by odd reflection with respect to the second variable. By construction, we have that this is a weak solution of 
\[
-\Delta u=\lambda_1(\mathbf{H})\,u,\qquad \mbox{ in } A'_\tau.
\]
Let us still denote it by $u_1$, for simplicity. 
By standard Elliptic Regularity (see \cite[Corollary 8.11]{GT}), we have that $u_1$ is actually $C^\infty(E)$ for every open set $E$ compactly contained in $A'_\tau$. This shows in particular that $u_1$ is $C^2$ on $[\tau,L]\times[\tau,1]$.
\vskip.2cm\noindent
Let us now take $U\in C^\infty_0(\mathbf{H}^\varepsilon_n)$, we observe that this is not a feasible test function for the equation of $u_1$. Since $U$ is compactly supported, there exists $0<\delta\ll 1$ such that its support is contained in
\[
\mathbf{H}_n^\varepsilon(\delta):=\Big\{x\in\mathbf{H}\, :\, \mathrm{dist}(x,\partial\mathbf{H}^\varepsilon_n)\ge \delta,\, x_1\le \frac{1}{\delta}\Big\}.
\]
By the previous point and the fact that $u_1\in C^\infty(\mathbf{H})$, we know that 
\[
u_1\in C^2\Big(\mathbf{H}_n^\varepsilon(\delta)\cap\{(x_1,x_2)\, :\, x_2\le 1\}\Big).
\]
In particular $u_1$ solves in classical sense the equation, on this set. By the Divergence Theorem, we thus have
\[
\begin{split}
\lambda_1(\mathbf{H})\,\int_{\mathbf{H}} u_1\,U\,dx=\lambda_1(\mathbf{H})\,\int_{\mathbf{H}_n^\varepsilon(\delta)\cap\{ x_2\le 1\}} u_1\,U\,dx&=-\int_{\mathbf{H}_n^\varepsilon(\delta)\cap\{ x_2\le 1\}} \Delta u_1\,U\,dx\\
&=-\int_{\mathbf{H}_n^\varepsilon(\delta)\cap\{ x_2\le 1\}} \mathrm{div}(\nabla u_1\,U)\,dx\\
&+\int_{\mathbf{H}_n^\varepsilon(\delta)\cap\{ x_2\le 1\}} \langle \nabla u_1,\nabla U\rangle\,dx\\
&=-\int_{\Gamma^\varepsilon_n} \frac{\partial u}{\partial x_2}\,U\,d\mathcal{H}^1+\int_{\mathbf{H}} \langle \nabla u_1,\nabla U\rangle\,dx.
\end{split}
\] 
This is the claimed identity, for a test function $U\in C^\infty_0(\mathbf{H}^\varepsilon_n)$. By a density argument (and using the continuity of the trace operator), we can then conclude that the identity holds for test functions in $W^{1,2}_0(\mathbf{H}^\varepsilon_n)$, as well.
\end{proof}
We now present some results, aimed at giving an asymptotic estimate for $\lambda_1(\mathbf{H}_n^\varepsilon)$, as $\varepsilon$ goes to $0$. We will crucially exploit the concept of {\it thin torsional rigidity}, see Appendix \ref{sec:A} below. 
In particular, we will need the following quantity 
\[
\mathbf{T}_n^\varepsilon:=\sup_{\varphi\in C^\infty_0(\mathbf{H}_n^\varepsilon)}\left\{2\,\int_{\Gamma_n^\varepsilon} \frac{\partial u_1}{\partial x_2}\, \varphi\,d\mathcal{H}^{1}-\int_{\mathbf{H}_n^\varepsilon}|\nabla \varphi|^2\,dx\right\},
\]
which is well-defined, in light of the construction and Lemma \ref{lm:regolaritaminchiona}. We then observe that, with the notation of Definition \ref{def:sottiletta}, we have
	\begin{equation}
	\label{cosi}
	\mathbf{T}_n^\varepsilon=\mathcal{T}_{\Omega_\Sigma}(\Sigma;f),\qquad \mbox{ with }\Omega=\mathbf{H},\ \Sigma=\Gamma_n^\varepsilon\ \mbox{ and }\ f=\frac{\partial u_1}{\partial x_2}.
	\end{equation}
The following two results will permit to quantify the decay rate to $0$ of this quantity, as $\varepsilon$ goes to $0$.
\begin{lemma}[Upper bound]
\label{lemma:big_O}
With the notation above, we have
	\begin{equation*}
	\mathbf{T}_n^\varepsilon\leq 2\, (n+1)\, \varepsilon^2\, \left\|\frac{\partial u_1}{\partial x_2}\right\|^2_{L^\infty([1,4]\times\{1\})},\quad \mbox{ for every } n\in\mathbb{N}\setminus\{0\} \mbox{ and } 0<\varepsilon\le \frac{1}{4\,n}.
	\end{equation*}
\end{lemma}
\begin{proof}
	First of all, by construction we have $\Gamma_n^\varepsilon\subseteq [1,4]\times\{1\}$. Thus we get
	\[
		\left\|\frac{\partial u_1}{\partial x_2}\right\|_{L^\infty(\Gamma_n^\varepsilon)}\leq  \left\|\frac{\partial u_1}{\partial x_2}\right\|_{L^\infty([1,4]\times\{1\})}.
	\]
	The last term is finite thanks to Lemma \ref{lm:regolaritaminchiona}.
By \eqref{cosi}, we can thus apply Lemma \ref{lemma:upper_bound_torsion}. 
This entails that 
	\[
	\begin{split}
	\mathbf{T}_n^\varepsilon&\leq \gamma(\mathbf{H}_n^\varepsilon)\,\left\|\frac{\partial u_1}{\partial x_2}\right\|^2_{L^\infty([1,4]\times\{1\})}\, \mathcal{H}^{1}(\Gamma_n^\varepsilon)=2\,\varepsilon\,(n+1)\,\gamma(\mathbf{H}_n^\varepsilon)\,\left\|\frac{\partial u_1}{\partial x_2}\right\|^2_{L^\infty([1,4]\times\{1\})},
	\end{split}
	\]
	where $\gamma(\mathbf{H}_n^\varepsilon)$ is the sharp trace--type constant defined in Lemma \ref{lemma:upper_bound_torsion}. We will show that $\gamma(\mathbf{H}_n^\varepsilon)\le \varepsilon$. By definition, we have
	\[
	\begin{split}
	\gamma(\mathbf{H}_n^\varepsilon)=\sup_{\varphi\in W^{1,2}_0(\mathbf{H}_n^\varepsilon)\setminus\{0\}}\frac{\displaystyle \int_{\Gamma_n^\varepsilon} |\varphi|^2\,d\mathcal{H}^{1}}{\displaystyle \int_{\mathbf{H}_n^\varepsilon} |\nabla \varphi|^2\,dx }&=\sup_{\varphi\in C^\infty_0(\mathbf{H}_n^\varepsilon)\setminus\{0\}}\frac{\displaystyle \int_{\Gamma_n^\varepsilon} |\varphi|^2\,d\mathcal{H}^{1}}{\displaystyle \int_{\mathbf{H}_n^\varepsilon} |\nabla \varphi|^2\,dx }\\
	&=\sup_{\varphi\in C^\infty_0(\mathbf{H}_n^\varepsilon)\setminus\{0\}}\frac{\displaystyle \sum_{i=0}^n\int_{\Sigma_i^\varepsilon\times\{1\}} |\varphi|^2\,d\mathcal{H}^{1}}{\displaystyle \int_{\mathbf{H}_n^\varepsilon} |\nabla \varphi|^2\,dx }\\
	\end{split}
	\]
We take $\varphi\in C^\infty_0(\mathbf{H}_n^\varepsilon)\setminus\{0\}$ and observe that 
\begin{equation}
\label{paura}
\mbox{ if } \int_{T_i^\varepsilon} |\nabla \varphi|^2\,dx=0, \quad \mbox{ then } \quad \int_{\Sigma_i^\varepsilon\times\{1\}} |\varphi|^2\,d\mathcal{H}^1=0.
\end{equation}
Indeed, the first condition implies that $\varphi$ is constant in the tube $T_i^\varepsilon$. Since $\varphi$ is compactly supported in $\mathbf{H}_n^\varepsilon$ and $\partial T_i^\varepsilon\cap \partial \mathbf{H}_n^\varepsilon\not=\emptyset$, this shows that $\varphi$ must identically vanish on the tube $\overline{T_i^\varepsilon}$. Finally, since $\Sigma_i^\varepsilon\times\{1\}\subseteq \partial T_i^\varepsilon$, we get that $\varphi$ identically vanishes on $\Sigma_i^\varepsilon\times\{1\}$, as well.  
\par
Thanks to \eqref{paura}, we can write
\[
\gamma(\mathbf{H}_n^\varepsilon)=\sup_{\varphi\in C^\infty_0(\mathbf{H}_n^\varepsilon)\setminus\{0\}}\left\{\frac{\displaystyle \sum_{i=0}^n\int_{\Sigma_i^\varepsilon\times\{1\}} |\varphi|^2\,d\mathcal{H}^{1}}{\displaystyle \int_{\mathbf{H}_n^\varepsilon} |\nabla \varphi|^2\,dx }\, :\, \exists i_0\in \{0,\dots,n\} \mbox{ such that } \int_{T^\varepsilon_{i_0}} |\nabla \varphi|^2\,dx\not=0\right\}.
\]
Let us take $\varphi\in C^\infty_0(\mathbf{H}_n^\varepsilon)\setminus\{0\}$ which is admissible for the maximization problem on the right-hand side. Let $J_\varphi\subseteq \{0,\dots,n\}$ be the set of indices such that
\[
\int_{T_i^\varepsilon} |\nabla \varphi|^2\,dx\not=0,\qquad \mbox{ for } i\in J_\varphi.
\]
In light of \eqref{paura} and the definition of $J_\varphi$ , we then have\footnote{We use the elementary inequality
	\[
		\frac{\displaystyle\sum_{i=1}^n a_i}{\displaystyle\sum_{i=1}^n b_i}\leq \max\left\{\frac{a_i}{b_i}\, :\,i=1,\dots,n\right\},\qquad \mbox{ for all }a_i\ge 0,b_i> 0.
		\]} 
\[
\begin{split}
\frac{\displaystyle \sum_{i=0}^n\int_{\Sigma_i^\varepsilon\times\{1\}} |\varphi|^2\,d\mathcal{H}^{1}}{\displaystyle \int_{\mathbf{H}_n^\varepsilon} |\nabla \varphi|^2\,dx }\le \frac{\displaystyle \sum_{i\in J_\varphi}\int_{\Sigma_i^\varepsilon\times\{1\}} |\varphi|^2\,d\mathcal{H}^{1}}{\displaystyle \sum_{i\in J_\varphi}\int_{T_i^\varepsilon} |\nabla \varphi|^2\,dx }&\le \max_{i\in J_\varphi} \frac{\displaystyle \int_{\Sigma_i^\varepsilon\times\{1\}} |\varphi|^2\,d\mathcal{H}^{1}}{\displaystyle \int_{T_i^\varepsilon} |\nabla \varphi|^2\,dx }\\
&\le \max_{i=1,\dots,n}\sup_{\varphi\in C^\infty_0(T_i^\varepsilon)\setminus\{0\}}\frac{\displaystyle \int_{\Sigma_i^\varepsilon\times\{1\}} |\varphi|^2\,d\mathcal{H}^{1}}{\displaystyle \int_{T_i^\varepsilon} |\nabla \varphi|^2\,dx }.
\end{split}
\]
Finally, by scaling and translating, we see that 
\[
\sup_{\varphi\in C^\infty_0(T_i^\varepsilon)\setminus\{0\}}\frac{\displaystyle \int_{\Sigma_i^\varepsilon} |\varphi|^2\,d\mathcal{H}^{1}}{\displaystyle \int_{T_i^\varepsilon} |\nabla \varphi|^2\,dx }=\varepsilon \sup_{\varphi\in C^\infty_0(\mathrm{I}\times[1,2))\setminus\{0\}}\frac{\displaystyle \int_{(-1,1)\times\{1\}} |\varphi|^2\,d\mathcal{H}^{1}}{\displaystyle \int_{(-1,1)\times[1,2)} |\nabla \varphi|^2\,dx}\le \varepsilon.
\]
In the last estimate we used the following trace inequality
\[
\begin{split}
\int_{\mathrm{I}\times[1,2)} |\nabla \varphi|^2\,dx\ge \int_\mathrm{I}\left(\int_{1}^2 \left|\frac{\partial \varphi}{\partial x_2}\right|^2\,dx_2\right)\,dx_1\ge \int_\mathrm{I} |\varphi(x_1,1)|^2\,dx_1=\int_{\mathrm{I}\times\{1\}} |\varphi|^2\,d\mathcal{H}^{1},
\end{split}
\]
which holds for every $\varphi\in C^\infty_0(\mathrm{I}\times [1,2))$. This discussion proves that $\gamma(\mathbf{H}_n^\varepsilon)\le \varepsilon$, as claimed.
\end{proof}

\begin{lemma}[Asymptotical lower bound]
With the notation above, there exists a universal constant $\alpha>0$ such that for every $n\in\mathbb{N}\setminus\{0\}$ we have
\label{lm:blow_up}
\[
\mathbf{T}_n^\varepsilon\ge \alpha\,\varepsilon^2\,\sum_{i=0}^n\left(\frac{\partial u_1}{\partial x_2}\left(2+\frac{i}{n},1\right)\right)^2+o(\varepsilon^2),\qquad \mbox{ as }\varepsilon\searrow 0.
\]
\end{lemma}
\begin{proof}
By using \eqref{cosi} and the super-additivity of the thin torsional rigidity (see Lemma \ref{lemma:super}), we get
\[
\mathbf{T}_n^\varepsilon\ge \sum_{i=0}^n\mathcal{T}_{\mathbf{H}\cup T_i^\varepsilon}\left(\Sigma_i^\varepsilon\times\{1\};\frac{\partial u_1}{\partial x_2}\right).
\]
Observe that we used that $\partial u_1/\partial x_2$ has constant sign on $\Gamma^\varepsilon_n$.
On the other hand, as consequence of Theorem \ref{thm:blow_up_AO} applied with
\[
\mathbf{p}_i=\left(2+\frac{i}{n},1\right),\qquad \Sigma_i=(\mathrm{I}\times\{1\})+\mathbf{p}_i,\qquad f_i=\frac{\partial u_1}{\partial x_2},
\]
we know that for every $i=0,\dots,n$ we have
	\[
			\mathcal{T}_{\mathbf{H}\cup T_i^\varepsilon}\left(\Sigma_i^\varepsilon\times\{1\};\frac{\partial u_1}{\partial x_2}\right)=\alpha_i\,\left(\frac{\partial u_1}{\partial x_2}\left(2+\frac{i}{n},1\right)\right)^2\,\varepsilon^2+o(\varepsilon^2),\qquad\mbox{ as }\varepsilon\searrow 0,
	\]
	for a universal constant $\alpha_i>0$.  More precisely, with the notations of Definition \ref{def:torsion_blow_up}, the latter is given by
\[
\alpha_i=\mathfrak{T}_{\Pi_{\Sigma_i\times\{1\}}}(\Sigma_i\times\{1\})=\mathfrak{T}_{\Pi_{\mathrm{I}\times\{0\}}}(\mathrm{I}\times\{0\})>0.
\]
Observe in particular that $\alpha_i$ does not depend $i$, thanks to the fact that the quantity $\mathfrak{T}$ is invariant by rigid movements.	
This is enough to conclude.
	\end{proof}
The relevance of $\mathbf{T}_n^\varepsilon$ in our problem is encoded by the following estimate.	
\begin{proposition}[Eigenvalue estimate]\label{prop:eigen_estim}
With the notation above, for every $n\in\mathbb{N}\setminus\{0\}$ we have that
	\begin{equation}\label{eq:eigen_estim_th1}
\lambda_1(\mathbf{H}_n^\varepsilon)\leq \lambda_1(\mathbf{H})	-\mathbf{T}_n^\varepsilon+o(\mathbf{T}_n^\varepsilon),\qquad \mbox{ as }\varepsilon\searrow 0,
	\end{equation}
\end{proposition}
\begin{proof}
We first observe that $\lambda_1(\mathbf{H}_n^\varepsilon)>0$, since $\mathbf{H}_n^\varepsilon$ is bounded in the $x_2$ direction. By appealing again to \eqref{cosi}, we can apply Proposition \ref{prop:esisteU} and infer existence of $U_{\Gamma_n^\varepsilon}\in W^{1,2}_0(\mathbf{H}_n^\varepsilon)$ which attains the maximum in the definition of $\mathbf{T}_n^\varepsilon$.
\par
We use the trial function $u_1-U_{\Gamma_n^\varepsilon}$ in the variational problem which defines $\lambda_1(\mathbf{H}_n^\varepsilon)$. This gives
	\begin{equation}
	\label{eq:eigen_estim_1}
	\begin{split}
		\lambda_1(\mathbf{H}_n^\varepsilon)\leq \frac{\displaystyle \int_{\mathbf{H}_n^\varepsilon}|\nabla (u_1-U_{\Gamma_n^\varepsilon})|^2\,dx}{\displaystyle \int_{\mathbf{H}_n^\varepsilon}(u_1-U_{\Gamma_n^\varepsilon})^2\,dx}&=\frac{\displaystyle \lambda_1(\mathbf{H})+\int_{\mathbf{H}_n^\varepsilon}|\nabla U_{\Gamma_n^\varepsilon}|^2\,dx-2\,\int_{\mathbf{H}_n^\varepsilon}\langle\nabla u_1,\nabla U_{\Gamma_n^\varepsilon}\rangle\,dx }{\displaystyle 1+\int_{\mathbf{H}_n^\varepsilon}U_{\Gamma_n^\varepsilon}^2\,dx-2\int_{\mathbf{H}_n^\varepsilon}u_1 U_{\Gamma_n^\varepsilon}\,dx}\\
		&\le \frac{\displaystyle \lambda_1(\mathbf{H})+\int_{\mathbf{H}_n^\varepsilon}|\nabla U_{\Gamma_n^\varepsilon}|^2\,dx-2\,\int_{\mathbf{H}_n^\varepsilon}\langle\nabla u_1,\nabla U_{\Gamma_n^\varepsilon}\rangle\,dx }{\displaystyle 1-2\int_{\mathbf{H}_n^\varepsilon}u_1\, U_{\Gamma_n^\varepsilon}\,dx}.
	\end{split}
	\end{equation}
By using $u_1$ as a test function for the equation of $U_{\Gamma_n^\varepsilon}$	and observing that $u_1$ has a null trace on $\Gamma_n^\varepsilon$, we get
\begin{equation}
	\label{eq:eigen_estim_2}
0=\int_{\mathbf{H}^\varepsilon_n} \langle\nabla u_1,\nabla U_{\Gamma_n^\varepsilon}\rangle\,dx=\int_{\mathbf{H}} \langle\nabla u_1,\nabla U_{\Gamma_n^\varepsilon}\rangle\,dx.
\end{equation}
On the the other hand, by using $U_{\Gamma^\varepsilon_n}$ as a test function for the equation of $u_1$, one obtains from Lemma \ref{lm:regolaritaminchiona} that
	\begin{equation}
	\label{eq:eigen_estim_3}
	\begin{split}
		\int_{\mathbf{H}_n^\varepsilon}u_1\, U_{\Gamma_n^\varepsilon}\,dx&=\int_{\mathbf{H}} u_1\, U_{\Gamma_n^\varepsilon}\,dx \\
		&=\frac{1}{\lambda_1(\mathbf{H})}\,\int_{\mathbf{H}} \langle\nabla u_1,\nabla U_{\Gamma_n^\varepsilon}\rangle\,dx-\frac{1}{\lambda_1(\mathbf{H})}\int_{\Gamma_n^\varepsilon}U_{\Gamma_n^\varepsilon}\frac{\partial u_1}{\partial x_2} d\mathcal{H}^{1}=-\frac{1}{\lambda_1(\mathbf{H})}\,\mathbf{T}_n^\varepsilon.
	\end{split}
	\end{equation}
	In the last identity, we used \eqref{torsion_equiv}.
	Hence,  from \eqref{eq:eigen_estim_1}, \eqref{eq:eigen_estim_1} and \eqref{eq:eigen_estim_2} we get that
	\begin{equation*}
		\lambda_1(\mathbf{H}_n^\varepsilon)\leq \left(\lambda_1(\mathbf{H})+\mathbf{T}_n^\varepsilon\right)\,\left(1+\frac{2}{\lambda_1(\mathbf{H})}\,\mathbf{T}_n^\varepsilon\right)^{-1}.
	\end{equation*}
	For every fixed $n\in\mathbb{N}\setminus\{0\}$, the quanity $\mathbf{T}_n^\varepsilon$ is infinitesimal, as $\varepsilon$ goes to $0$ (thanks to Lemma \ref{lemma:big_O}). Thus we obtain that
	\[
	\begin{split}
		\lambda_1(\mathbf{H}_n^\varepsilon)&\leq (\lambda_1(\mathbf{H})+\mathbf{T}_n^\varepsilon)\,\left(1-\frac{2}{\lambda_1(\mathbf{H})}\,\mathbf{T}_n^\varepsilon+o(\mathbf{T}_n^\varepsilon)\right),\qquad \mbox{ as } \varepsilon\searrow 0.
	\end{split}
	\]
This proves the claimed asymptotical estimate \eqref{eq:eigen_estim_th1}.
\end{proof}

The main outcome of the previous discussion is contained in the following result. This simply follows by combining Proposition \ref{prop:eigen_estim}, Lemma \ref{lemma:big_O} and Lemma \ref{lm:blow_up}.

\begin{corollary}\label{cor:eigen_H_eps}
	For any fixed $n\in\mathbb{N}$, there holds
	\begin{equation*}
		\lambda_1(\mathbf{H}_n^\varepsilon)\leq \lambda_1(\mathbf{H})-\varepsilon^2\,\alpha\,\sum_{i=0}^n\left(\frac{\partial u_1}{\partial x_2}\left(2+\frac{i}{n},1\right)\right)^2+o(\varepsilon^2),\qquad\text{as }\varepsilon\searrow 0.
	\end{equation*}
\end{corollary}
Finally, by recalling the notation \eqref{rhoberto}, we can prove the following
\begin{theorem}
With the notation above,	there exist $n_0\in\mathbb{N}\setminus\{0\}$ such that
	\begin{equation*}
		\rho(\mathbf{H}^\varepsilon_{n_0})\le \left(1-\varepsilon^2\right)\,\rho(\mathbf{H})+o(\varepsilon^2),\qquad\text{ as }\varepsilon \searrow 0.
	\end{equation*} 
	In particular, $\mathbf{H}$ does not provide the sharp Makai-Hayman constant. 
\end{theorem}
\begin{proof}
We first decide the number $n_0$ of tubes to be attached. At this aim, we observe that the following constant
\[
\mathfrak{m}:=\min_{x_1\in[1,4]} \left(\frac{\partial u_1}{\partial x_2}(x_1,1)\right)^2,
\]
is positive, thanks to the Hopf Boundary Lemma. 
Accordingly, we set 
	\[
	n_0:=\min\left\{n\in\mathbb{N}\setminus\{0\}\, :\, n\ge \frac{\lambda_1(\mathbf{H})\,(1+2\,r_{\mathbf{H}}^2)}{2\,\alpha\,\mathfrak{m}\,r_\mathbf{H}^2}\right\},
	\]
	and observe that such a choice is universal.
	Thanks to this definition, we have 
	\begin{equation}
		\label{rompi}
		\alpha\,\sum_{i=0}^{n_0}\left(\frac{\partial u_1}{\partial x_2}\left(2+\frac{i}{n_0},1\right)\right)^2\ge \frac{\lambda_1(\mathbf{H})\,(1+2\,r_{\mathbf{H}}^2)}{2\,r_{\mathbf{H}}^2}.
	\end{equation}
	In order to estimate the quantity $\rho(\mathbf{H}^{n_0}_\varepsilon)$, one can easily check that
	\begin{equation*}
		r_{\mathbf{H}^\varepsilon_{n_0}}=\frac{1}{2}+\frac{\varepsilon^2}{2}=r_{\mathbf{H}}+\frac{\varepsilon^2}{4\,r_{\mathbf{H}}},
	\end{equation*}
	which implies that 
	\begin{equation*}
		r_{\mathbf{H}^\varepsilon_{n_0}}^2=r_{\mathbf{H}}^2+\frac{\varepsilon^2}{2}+o(\varepsilon^2),\quad\text{as }\varepsilon\searrow 0.
	\end{equation*}
	Observe that the variation on the inradius does not depend on $n_0$. This is the crucial point.	
	Thanks to this fact and to Corollary \ref{cor:eigen_H_eps}, we have that, as $\varepsilon$ goes to $0$,
	\[
	\begin{split}
		\rho(\mathbf{H}^\varepsilon_{n_0})=r_{\mathbf{H}^\varepsilon_{n_0}}^2\,\lambda_1(\mathbf{H}^\varepsilon_{n_0})&\leq \left(r_{\mathbf{H}}^2+\frac{\varepsilon^2}{2}+o(\varepsilon^2)\right)\,\left(\lambda_1(\mathbf{H})-\varepsilon^2\,\alpha\,\sum_{i=0}^{n_0}\left(\frac{\partial u_1}{\partial x_2}\left(2+\frac{i}{n_0},1\right)\right)^2+o(\varepsilon^2)\right) \\
		&=\rho(\mathbf{H})+\varepsilon^2\,\left(\frac{\lambda_1(\mathbf{H})}{2}- r_{\mathbf{H}}^2\,\alpha\,\sum_{i=0}^{n_0}\left(\frac{\partial u_1}{\partial x_2}\left(2+\frac{i}{n_0},1\right)\right)^2\right)+o(\varepsilon^2).
	\end{split}
	\]
By recalling \eqref{rompi}, we get the desired conclusion.
\end{proof}

\appendix 

\section{The thin torsional rigidity}
\label{sec:A}

In this section, we briefly recall some facts about the {\it thin torsional rigidity} used in Section \ref{sec:7}, essentially taken from \cite{AO2023}. The main difference with \cite{AO2023} consists in the fact that we want to allow many tubes to be attached, rather than just one. 
\par
We start by describing the class of sets which will be singularly perturbed by thin tubes.
\begin{assumptions}\label{ass:domain}
For $N\ge 2$, let $\Omega\subseteq\mathbb{R}^N$ be an open set with the following property. For some $M\in\mathbb{N}\setminus\{0\}$, there exist:
\[
\mathbf{p}_1,\dots,\mathbf{p}_M\in\partial\Omega,\qquad \nu_1,\dots,\nu_M\in\mathbb{S}^{N-1},\qquad r_1,\dots,r_M>0,
\]
such that, for all $i=1,\dots,M$, there holds:
\begin{enumerate}
		\item[i.] $\partial\Omega\cap B_{r_i}(\mathbf{p}_i)=\partial\Omega\cap\{x\in B_{r_i}(\mathbf{p}_i)\,:\,  \langle x-\mathbf{p}_i,\nu_i\rangle=0\}$;
	\vskip.2cm
	\item[ii.] $(\mathbb{R}^N\setminus\overline{\Omega})\cap B_{r_i}(\mathbf{p}_i)=(\mathbb{R}^N\setminus\overline{\Omega})\cap\{x\in B_{r_i}(\mathbf{p}_i)\, :\, \langle x-\mathbf{p}_i,\nu_i\rangle>0\}$.
\end{enumerate}
In particular, the boundary of $\Omega$ is flat around each point $\mathbf{p}_i$, with $\nu_i$ representing the unit outer normal vector of $\partial\Omega$ around this point.
\par
Moreover, for every $i=1,\dots,M$ we take a relatively open connected subset $\Sigma_i\subseteq \partial\Omega\cap B_{r_i}(\mathbf{p}_i)$, such that $\mathbf{p}_i\in \Sigma_i$ and $\Sigma_i\cap \Sigma_j=\emptyset$ for $i\not=j$. Then we define
\begin{equation*}
	\Sigma:=\bigcup_{i=1}^M\Sigma_i.
\end{equation*}
For a pair $(\Omega,\Sigma)$ as above, we define
\begin{equation*}
	T_{\Sigma_i}:=\{x+t\,\nu_i\,:\, x\in\Sigma_i,\,  t\in[0,1)\}\qquad \mbox{ and }\qquad\Omega_\Sigma:=\Omega \cup \bigcup_{i=1}^M T_{\Sigma_i},
\end{equation*}
and we also assume that
\[
	T_{\Sigma_i}\cap\Omega=\emptyset\qquad \mbox{and}\qquad	T_{\Sigma_i}\cap T_{\Sigma_j}=\emptyset,
\]
for all $i,j=1,\dots,M$ such that $i\neq j$.
\end{assumptions}
 
\begin{definition}
\label{def:sottiletta}
	Let $\Omega\subseteq\mathbb{R}^N$ and $\Sigma\subseteq\partial\Omega$ be a pair satisfying Assumptions \ref{ass:domain}. For any $f\in L^2(\Sigma)$, we define the \emph{thin $f-$torsional rigidity} of $\Sigma$ relative to $\Omega_\Sigma$ as follows
	\begin{equation*}
		\mathcal{T}_{\Omega_\Sigma}(\Sigma;f):=\sup_{\varphi\in C^\infty_0(\Omega_\Sigma)}\left\{2\,\int_\Sigma f\, \varphi\,d\mathcal{H}^{N-1}-\int_{\Omega_\Sigma}|\nabla \varphi|^2\,dx\right\}.
	\end{equation*}
	Of course, the supremum is unchanged if settled on $W^{1,2}_0(\Omega_\Sigma)$.
In the particular case $f\equiv 1$, we will simply write
	\begin{equation*}
		\mathcal{T}_{\Omega_\Sigma}(\Sigma):=\mathcal{T}_{\Omega_\Sigma}(\Sigma;1),
	\end{equation*}
	and we call it the \emph{thin torsional rigidity} of $\Sigma$ relative to $\Omega_\Sigma$.
\end{definition}

\begin{proposition}
\label{prop:esisteU}
		Let $\Omega\subseteq\mathbb{R}^N$ and $\Sigma\subseteq\partial\Omega$ be a pair satisfying Assumptions \ref{ass:domain}. Let us suppose that $\lambda_1(\Omega_\Sigma)>0$.
		For any $f\in L^2(\Sigma)$, we have
	\[
	\mathcal{T}_{\Omega_\Sigma}(\Sigma;f)=	\max_{\varphi\in W^{1,2}_0(\Omega_\Sigma)}\left\{2\,\int_\Sigma f\, \varphi\,d\mathcal{H}^{N-1}-\int_{\Omega_\Sigma}|\nabla \varphi|^2\,dx\right\},
	\]	
and this is a positive quantity.	
Moreover,	such a maximum is (uniquely) attained by a function $U_{\Sigma,f}\in W^{1,2}_0(\Omega_\Sigma)$, which weakly satisfies
	\[			
	\left\{\begin{array}{rccl}
	-\Delta U_{\Sigma,f} &=&0, &\text{in }\Omega_\Sigma\setminus\Sigma, \\
	&&&\\
			\dfrac{\partial U_{\Sigma,f}}{\partial \nu_i^+ }+\dfrac{\partial U_{\Sigma,f}}{\partial \nu_i^-} &=&f, &\text{on }\Sigma_i, \mbox{ for } i=1,\dots,M,
			\end{array}
			\right.
	\]
	where 
	\[
	\frac{\partial U_{\Sigma,f}}{\partial \nu_i^\pm}(x):=\lim_{h\to 0^-}\frac{U_{\Sigma,f}(x\pm h\,\nu_i)-U_{\Sigma,f}(x)}{h},\qquad \mbox{ for }x\in \Sigma_i.
	\] 
		More precisely, there holds
	\begin{equation*}
		\int_{\Omega_\Sigma}\langle\nabla U_{\Sigma,f},\nabla \varphi\rangle\,dx=\int_\Sigma \varphi\,f \,d\mathcal{H}^{N-1},\qquad \mbox{ for every }\varphi\in W^{1,2}_0(\Omega_\Sigma).
	\end{equation*} 
	Finally, we have 
	\begin{equation}
	\label{torsion_equiv}
	\mathcal{T}_{\Omega_\Sigma}(\Sigma;f)=\int_\Sigma f\, U_{\Sigma,f}\,d\mathcal{H}^{N-1}.
	\end{equation}
\end{proposition}
\begin{proof}
The assumption $\lambda_1(\Omega_\Sigma)>0$ entails that 
\[
\varphi\mapsto \left(\int_{\Omega_\Sigma}|\nabla \varphi|^2\,dx\right)^\frac{1}{2},
\]
is an equivalent norm on $W^{1,2}_0(\Omega_\Sigma)$. Moreover, we have the continuity of the trace operator
\[
\mathrm{Tr}:W^{1,2}_0(\Omega_\Sigma)\to L^2\left(\bigcup_{i=1}^M\big(\partial\Omega\cap B_{r_i}(\mathbf{p}_i)\big)\right),
\]
see for example \cite[Theorem 18.40]{Leo}.
Then the proof of this result is standard, it is sufficient to use the Direct Method in the Calculus of Variations.
\end{proof}
\begin{lemma}[Superadditivity]
\label{lemma:super}
	Let $\Omega\subseteq\mathbb{R}^N$ and $\Sigma\subseteq\partial\Omega$ be a pair satisfying Assumptions \ref{ass:domain}. Then for every $f\in L^2(\Sigma)$ we have
	\[
\mathcal{T}_{\Omega_\Sigma}(\Sigma;|f|)=\mathcal{T}_{\Omega_\Sigma}(\Sigma;-|f|),	
	\]	
and
\begin{equation*}
\mathcal{T}_{\Omega_\Sigma}(\Sigma;|f|)\geq \sum_{i=1}^M\mathcal{T}_{\Omega\cup T_{\Sigma_i}}(\Sigma_i;|f|).
\end{equation*}
\end{lemma}
\begin{proof}
The first fact simply follows by observing that
\[
2\,\int_\Sigma |f|\, \varphi\,d\mathcal{H}^{N-1}-\int_{\Omega_\Sigma}|\nabla \varphi|^2\,dx=2\,\int_\Sigma (-|f|)\, (-\varphi)\,d\mathcal{H}^{N-1}-\int_{\Omega_\Sigma}|\nabla (-\varphi)|^2\,dx.
\]
For every $i=1,\dots,M$, let $U_i\in W^{1,2}_0(\Omega\cup T_{\Sigma_i})$ be a function admissible for the optimization problem defining $\mathcal{T}_{\Omega\cup T_{\Sigma_i}}(\Sigma_i)$. We extend each $U_i$ to the whole $\Omega_\Sigma$, by defining it to be identically $0$ in $T_{\Sigma_j}$, with $j\not=i$. We then observe that the trial function
	\[
	U=\max\Big\{|U_1|,\dots,|U_M| \Big\},
	\]
	belongs to $W^{1,2}_0(\Omega_\Sigma)$. Thus, we get
	\[
	\begin{split}
\mathcal{T}_{\Omega_\Sigma}(\Sigma;|f|)&\ge 2\,\int_\Sigma |f|\, U\,d\mathcal{H}^{N-1}-\int_{\Omega_\Sigma}|\nabla U|^2\,dx\\
&\ge 2\,\sum_{i=1}^M\int_{\Sigma_i} |f|\, U\,d\mathcal{H}^{N-1}-\sum_{i=1}^M\int_{\Omega_\Sigma}|\nabla U_i|^2\,dx,	
	\end{split}
	\]
	thanks to the fact that 
	\[
	\int_{\Omega_\Sigma} |\nabla U|^2\,dx\le \sum_{i=1}^M \int_{\Omega_\Sigma} \big|\nabla |U_i|\big|^2\,dx=\sum_{i=1}^M \int_{\Omega_\Sigma} \big|\nabla U_i\big|^2\,dx.
	\]
Moreover, by construction we have $U=|U_i|\ge U_i$ on $\Sigma_i$,
which permits to infer that
\[
	\begin{split}
\mathcal{T}_{\Omega_\Sigma}(\Sigma;|f|)\ge \sum_{i=1}^M \left[2\,\int_{\Sigma_i} |f|\, U_i\,d\mathcal{H}^{N-1}-\int_{\Omega_\Sigma}|\nabla U_i|^2\,dx\right],	
	\end{split}
\]
By arbitrariness of the functions $U_i$, this gives the desired result.
\end{proof}

\begin{lemma}\label{lemma:upper_bound_torsion}
	Let $\Omega\subseteq\mathbb{R}^N$ and $\Sigma\subseteq\partial\Omega$ be a pair satisfying Assumptions \ref{ass:domain}. Then, if $f\in L^\infty(\Sigma)$, there holds
	\begin{equation*}
		\mathcal{T}_{\Omega_\Sigma}(\Sigma;f)\leq \gamma(\Omega_\Sigma)\,\|f\|_{L^\infty(\Sigma)}^2\, \mathcal{H}^{N-1}(\Sigma),
	\end{equation*}
	where
	\begin{equation*}
		\gamma(\Omega_\Sigma):=\sup_{\substack{\varphi\in W^{1,2}_0(\Omega_\Sigma) \\ \varphi\neq 0}}\frac{\displaystyle \int_\Sigma |\varphi|^2\,d\mathcal{H}^{N-1}}{\displaystyle \int_{\Omega_\Sigma} |\nabla \varphi|^2\,dx }.
	\end{equation*}
\end{lemma}
\begin{proof}
	It is not difficult to see that $\mathcal{T}_{\Omega_\Sigma}(\Sigma;f)$ can be equivalently defined as
	\begin{equation*}
		\mathcal{T}_{\Omega_\Sigma}(\Sigma;f)=\sup_{\substack{\varphi\in W^{1,2}_0(\Omega_\Sigma) \\ \varphi\neq 0}}\frac{\displaystyle \left(\int_{\Sigma} \varphi\,f\, d\mathcal{H}^{N-1} \right)^2 }{\displaystyle \int_{\Omega_\Sigma}|\nabla \varphi|^2\,dx},
	\end{equation*}
	see \cite[Lemma 2.1]{AO2023}.	
	We conclude by applying the Cauchy-Schwarz inequality to the numerator and using that $f\in L^\infty(\Sigma)$.
\end{proof}

For any open set $E\subseteq\mathbb{R}^N$, we define the space $\mathcal{D}_0^{1,2}(E)$ as the completion of $C_0^\infty(E)$ with respect to the norm
\begin{equation}
\label{normaD}
	\|\varphi\|_{\mathcal{D}^{1,2}_0(E)}:=\left(\int_E|\nabla \varphi|^2\,dx\right)^{\frac{1}{2}},\qquad \mbox{ for every } \varphi\in C^\infty_0(E).
\end{equation}
\begin{remark}
\label{remark:spazi}
We recall that for $N\ge 3$, the space $\mathcal{D}^{1,2}_0(E)$ can be identified with the closure of $C^\infty_0(E)$ in the Banach space
\[
X^{2^*,2}(E):=\Big\{\varphi\in L^{2^*}(E)\, :\, \nabla \varphi\in L^2(E)\Big\},\qquad \mbox{ with } 2^*=\frac{2\,N}{N-2},
\]
endowed with the natural norm
\[
\|\varphi\|_{X^{2^*,2}(E)}:=\|\varphi\|_{L^{2^*}(E)}+\|\nabla \varphi\|_{L^2(E)}.
\]
This is possible thanks to Sobolev's inequality, which makes the two norms $\|\cdot\|_{X^{2^*,2}(E)}$ and $\|\cdot\|_{\mathcal{D}^{1,2}(E)}$ equivalent on $C^\infty_0(E)$.
\par
The case $N=2$ is more delicate, since we can not appeal to Sobolev's inequality. If $E\subsetneq\mathbb{R}^2$ is {\it simply connected}, we have at our disposal the following Hardy inequality
\[
\frac{1}{16}\,\int_E \left|\frac{\varphi}{d_E}\right|^2\,dx\le \int_E |\nabla \varphi|^2\,dx,\qquad \mbox{ for every } \varphi\in C^\infty_0(E),
\]
see \cite[page 278]{An} and also \cite{LS}. Here $d_E:E\to \mathbb{R}$ is the distance from the boundary $\partial E$. In light of this result, $\mathcal{D}^{1,2}_0(E)$ can be identified with 
the closure of $C^\infty_0(E)$ in the Hilbert space
\[
\dot W^{1,2}(E;d_E)=\left\{\varphi\in L^2_{\rm loc}(E)\, :\, \frac{\varphi}{d_E}\in L^2(E),\, \nabla \varphi\in L^2(E)\right\},
\]
endowed with the norm
\[
\|\varphi\|_{\dot W^{1,2}(E;d_E)}=\left(\left\|\frac{\varphi}{d_E}\right\|^2_{L^2(E)}+\|\nabla \varphi\|^2_{L^2(E)}\right)^\frac{1}{2}.
\]
\end{remark}

\begin{definition}\label{def:torsion_blow_up}
Let $\mathbf{p}\in\mathbb{R}^N$ and $\nu\in\mathbb{S}^{N-1}$ be fixed. For any $\Sigma \subseteq \{x\in\mathbb{R}^N\, :\, \langle x-\mathbf{p},\nu\rangle=0\}$ relatively open bounded set, we let 
	\[
	T_\Sigma^\infty:=\{x+t\,\nu\,:\, x\in\Sigma,\,  t\in(-\infty,0]\}\quad \mbox{ and }\quad \Pi_\Sigma:=\{x\in\mathbb{R}^N\, :\, \langle x-\mathbf{p},\nu\rangle> 0\}\cup T_\Sigma^\infty.
	\] 
	For any $f\in L^2(\Sigma)$, we denote
	\begin{equation*}
		\mathfrak{T}_{\Pi_\Sigma}(\Sigma;f):=\sup_{\varphi\in C^\infty_0(\Pi_\Sigma)}\left\{2\,\int_\Sigma f\, \varphi\,d\mathcal{H}^{N-1}-\int_{\Pi_\Sigma} |\nabla \varphi|^2\,dx \right\}.
	\end{equation*}
	Moreover, we denote
	\begin{equation*}
		\mathfrak{T}_{\Pi_\Sigma}(\Sigma):=\mathfrak{T}_{\Pi_\Sigma}(\Sigma;1).
	\end{equation*}
We observe in particular that this quantity is invariant by rigid movements.
\end{definition}

\begin{proposition}
For any $\Sigma \subseteq \{x\in\mathbb{R}^N\, :\, \langle x-\mathbf{p},\nu\rangle=0\}$ relatively open bounded set and any $f\in L^2(\Sigma)$, we have
\[
\mathfrak{T}_{\Pi_\Sigma}(\Sigma;f)=\max_{\varphi\in \mathcal{D}_0^{1,2}(\Pi_\Sigma)}\left\{2\,\int_\Sigma f\, \varphi\,d\mathcal{H}^{N-1}-\int_{\Pi_\Sigma} |\nabla \varphi|^2\,dx \right\},
\]
and this is a positive quantity.
Moreover,	such a maximum is (uniquely) attained by a function $\widetilde{U}_{\Sigma,f}\in \mathcal{D}_0^{1,2}(\Pi_\Sigma)$, which weakly satisfies
	\[
	\left\{\begin{array}{rccl}	
			-\Delta \widetilde{U}_{\Sigma,f} &=&0, &\text{ in }\Pi_\Sigma \setminus\Sigma, \\
			&&&\\
			\dfrac{\partial \widetilde{U}_{\Sigma,f}}{\partial \nu^+ }+\dfrac{\partial \widetilde{U}_{\Sigma,f}}{\partial \nu^-} &=&f, &\text{ on }\Sigma,
	\end{array}
	\right.
	\]
	where 
	\[
	\frac{\partial \widetilde U_{\Sigma,f}}{\partial \nu^\pm}(x):=\lim_{h\to 0^-}\frac{\widetilde U_{\Sigma,f}(x\pm h\,\nu)-\widetilde U_{\Sigma,f}(x)}{h},\qquad \mbox{ for }x\in \Sigma.
	\] 
More precisely, there holds
	\begin{equation*}
		\int_{\Pi_\Sigma}\langle \nabla \widetilde{U}_{\Sigma,f},\nabla \varphi\rangle\,dx=\int_\Sigma \varphi\,f \,d\mathcal{H}^{N-1},\qquad\mbox{ for every }\varphi\in \mathcal{D}^{1,2}(\Pi_\Sigma).
	\end{equation*}
\end{proposition}
\begin{proof}
Without loss of generality, we can suppose that $\mathbf{p}=0$ and $\nu=\mathbf{e}_N$. In this case, $\Sigma\subseteq\mathbb{R}^{N-1}$ and 
	\[
	T_\Sigma^\infty:=\Sigma\times (-\infty,0]\quad \mbox{ and }\quad \Pi_\Sigma:=\{x\in\mathbb{R}^N\, :\, x_N> 0\}\cup T_\Sigma^\infty.
	\] 
We need at first to show the existence of a continuous trace operator
\[
\mathrm{Tr}:\mathcal{D}^{1,2}_0(\Pi_\Sigma)\to L^2(\Sigma).
\]
We take $\varphi\in C^\infty_0(\Pi_\Sigma)$ and write 
\[
\varphi(x',0)=\varphi(x',x_N)+\int_0^{x_N} \frac{\partial \varphi}{\partial x_N}(x',\tau)\,d\tau,\qquad \mbox{ for every }x'\in\Sigma \mbox{ and } x_N\in(0,1).
\]
In particular, by raising to the power $2$ and using Jensen's inequality, we get
\[
\begin{split}
|\varphi(x',0)|^2&\le 2\,|\varphi(x',x_N)|^2+2\,x_N\,\int_0^{x_N}\left|\frac{\partial \varphi}{\partial x_N}(x',\tau)\right|^2\,d\tau\\
&\le 2\,|\varphi(x',x_N)|^2+2\,x_N\,\int_0^{1}\left|\frac{\partial \varphi}{\partial x_N}(x',\tau)\right|^2\,d\tau.
\end{split}
\]
We now integrate this estimate on $\Sigma\times(0,1)$, so to get
\begin{equation}
\label{pretraccia}
\int_\Sigma|\varphi(x',0)|^2\,dx'\le 2\,\int_{\Sigma\times(0,1)} |\varphi|^2\,dx+2\,\int_{\Sigma\times(0,1)} \left|\frac{\partial \varphi}{\partial x_N}\right|^2\,dx.
\end{equation}
If $N\ge 3$, we can estimate
\[
\int_{\Sigma\times(0,1)} |\varphi|^2\,dx\le |\Sigma\times(0,1)|^{1-\frac{2}{2^*}}\,\left(\int_{\Sigma\times(0,1)} |\varphi|^{2^*}\,dx\right)^\frac{2}{2^*}\le \frac{|\Sigma\times(0,1)|^{1-\frac{2}{2^*}}}{T_N}\,\int_{\Pi_\Sigma} |\nabla \varphi|^2\,dx,
\]
by Sobolev's inequality. If $N= 2$, we observe that $\Pi_\Sigma$ is a proper simply connected subset of $\mathbb{R}^2$ and use the Hardy inequality recalled in Remark \ref{remark:spazi} above. This permits to infer that
\[
\int_{\Sigma\times(0,1)} |\varphi|^2\,dx=\int_{\Sigma\times(0,1)} \left|\frac{\varphi}{d_{\Pi_\Sigma}}\right|^2\,d_{\Pi_\Sigma}^2\,dx\le C_\Sigma\,\int_{\Sigma\times(0,1)} \left|\frac{\varphi}{d_{\Pi_\Sigma}}\right|^2\le 16\,C_\Sigma\,\int_{\Pi_\Sigma} |\nabla \varphi|^2\,dx.
\]
We used that $d_{\Pi_\Sigma}$ is bounded by a constant $C_\Sigma$ on $\Sigma\times(0,1)$. Thus, both for $N\ge 3$ and $N=2$, from \eqref{pretraccia} we get 
\[
\left(\int_\Sigma |\varphi(x',0)|^2\,dx'\right)^\frac{1}{2}\le C\,\left(\int_{\Pi_\Sigma}|\nabla \varphi|^2\,dx\right)^\frac{1}{2},\qquad \mbox{ for every } \varphi\in C^\infty_0(\Pi_\Sigma),
\]
for some $C=C(N,\Sigma)>0$. This estimate shows in particular that for every Cauchy sequence $\{\varphi_n\}_{n\in\mathbb{N}}\subseteq C^\infty_0(\Pi_\Sigma)$ with respect to the norm \eqref{normaD}, we have that $\{\varphi(\cdot,0)\}_{n\in\mathbb{N}}$ is a Cauchy sequence in the Banach space $L^2(\Sigma)$, as well. This permits to define the trace operator in a standard way, which is thus continuous.
\par
The existence of a maximizer for $\mathfrak{T}_{\Pi_\Sigma}(\Sigma;f)$ can now be proved by means of the Direct Method in the Calculus of Variations. We leave the details to the reader.
\end{proof}

Finally, we recall \cite[Theorem 2.8]{AO2023}, which contains a blow-up analysis for the thin torsional rigidity of shrinking sets, here suitably stated in our framework. We point out that \cite[Theorem 2.8]{AO2023} has been proved only in dimension $N\geq 3$, but the very same argument can be repeated in dimension $2$. It is sufficient to use the characterization of $\mathcal{D}^{1,2}_0$ on simply connected proper subsets of $\mathbb{R}^2$, contained in Remark \ref{remark:spazi}. 
\par
In order to state it clearly for our needs, we need some notations more: let $\Omega\subseteq\mathbb{R}^N$ and $\Sigma\subseteq\partial\Omega$ be a pair satisfying Assumptions \ref{ass:domain}. For every $i=1,\dots,M$ and for every $0<\varepsilon<1$, we set
\[
\Sigma_i^\varepsilon:=\varepsilon\,(\Sigma_i-\mathbf{p}_i)+\mathbf{p}_i.
\]
Recall that $\mathbf{p}_i\in \Sigma_i$, thus $\Sigma_i^\varepsilon$ is shrinking to this point, as $\varepsilon$ goes to $0$. Then we set
\[
T_i^\varepsilon:=T_{\Sigma_i^\varepsilon}=\{x+t\,\nu_i\,:\, x\in\Sigma_i^\varepsilon,\,  t\in[0,1)\}\qquad \mbox{ and }\qquad\Omega_{\Sigma^\varepsilon_i}:=\Omega \cup  T_i^\varepsilon.
\]

\begin{theorem}\label{thm:blow_up_AO}
With the notations above, let us suppose that
\[
\Omega\subseteq \{x\in\mathbb{R}^N\, :\, \langle x-\mathbf{p}_i,\nu_i\rangle \le 0\}.
\]
Let $f_i\in C^0(\partial\Omega\cap B_{r_i}(\mathbf{p}_i))$ be such that $f_i(\mathbf{p}_i)\neq 0$. Then, we have
	\[
	\mathcal{T}_{\Omega\cup T^\varepsilon_i}(\Sigma^\varepsilon_i;f_i)=\varepsilon^2\, \Big( f_i(\mathbf{p}_i)\Big)^2\,\mathfrak{T}_{\Pi_{\Sigma_i}}(\Sigma_i)+o(\varepsilon^2),\quad\text{as }\varepsilon\searrow 0.
	\]
\end{theorem}

\end{document}